\documentclass[12pt]{article}

\usepackage{amstext    }
\usepackage{amsthm    }
\usepackage{a4}
\usepackage[mathscr]{eucal}
\usepackage{mathrsfs}

\usepackage{amsmath}
\usepackage{amssymb}
\usepackage{amscd}

\newtheorem{thmspec}{\relax}
\newtheorem{theorem}{Theorem}[section]
\newtheorem{definition}[theorem]{Definition}
\newtheorem{proposition}[theorem]{Proposition}
\newtheorem{corollary}[theorem]{Corollary}
\newtheorem{lemma}[theorem]{Lemma}
\newtheorem{remark}[theorem]{Remark}

\def\arg{\operatorname{arg}}
 \def\mes{\operatorname{mes}}
\newcommand {\Alim}{{\rm \mathcal{A}-{\rm \lim}}}
\def\Alimsup{{\rm \mathcal{A}{\rm -\limsup}}}
 \def\X{\Bbb X}

\newcommand{\cali}[1]{\mathscr{#1}}

\newcommand{\dist}{{\rm dist\ \!}}

\def\End{\operatorname{End}}

\newcommand{\Cc}{\cali{C}}

\newcommand{\Oc}{\cali{O}}

\newcommand{\C}{\mathbb{C}}

\title
{Cross theorems with singularities}

\author{Vi\^et-Anh Nguy{\^e}n and Peter Pflug}

\begin{document}

\maketitle

 \begin{abstract}
We establish  extension theorems for  separately holomorphic mappings defined on  sets of the form $W\setminus M$
with values in a complex analytic space which possesses the Hartogs extension property. Here  $W$ is   a $2$-fold
cross of arbitrary complex manifolds   and $M$ is a set of  singularities which is locally pluripolar  (resp. thin) in
fibers.
\end{abstract}

\noindent {\bf Classification AMS 2000:} Primary 32D15, 32D10


\section{Introduction}

Let $D\subset X$ (resp. $G\subset Y$) be an open set, $A\subset \overline{D}$ (resp. $B\subset \overline{G}$), where
$X$ and $Y$ are complex manifolds \footnote{ In this paper complex manifolds are always assumed
    to be of finite dimension and countable at infinity.},
 and  let $M\subset \big((D\cup A)\times
B\big)\bigcup \big(A\times (G\cup B)\big). $
 The set $M_a:=\{w\in G:\ (a,w)\in M\}$,  $a\in A$, is  called
 \textit{the vertical fiber of  $M$ over  $a$} (resp. the set
$M^{b}:=\{z\in D:\ (z,b)\in M\}$, $b\in B$, is  called \textit{the horizontal fiber  of $M$  over $b$}). We say that
$M$ possesses a certain property    {\it in fibers over  $A$   (resp. $B$)} if all vertical fibers $M_a$, $a\in A$,
(resp. all horizontal fibers $M^b$, $b\in B$) possess  this property.

\smallskip

The main purpose of this work is to  study  the following PROBLEM:

\medskip
{\it
 Let $X$, $Y$, $D$, $G$, $A$, and $B$ be as above,
and let  $Z$ be a complex analytic space~\footnote{All complex analytic spaces are assumed to be reduced, irreducible,  and countable at infinity.}\!.  Define the cross
\begin{equation*}
 W:=\big((D\cup A)\times B\big) \bigcup      \big(A\times (G\cup B)\big) .
 \end{equation*}

 We want to determine  an  ``optimal" open subset  of $ X\times Y,$ denoted by  $\widehat{\widetilde{W}},$ which
 is characterized by the following property:

 Let $M\subset W$ be  a  subset  which is  relatively closed and  locally  pluripolar   (resp. thin)
\footnote{    The notion  of local pluripolarity and  thinness  will be  recalled  in Section
\ref{Section_preliminary} below.} in fibers  over $A$  and $B$ ($M=\varnothing$ is  allowed). Then  there exists  a
new  set of singularities $\widehat{M} \subset \widehat{\widetilde{W}},$ which is, in some sense, of the  same
structure as $M$ and which
 fulfills the following property:

 For every mapping $f:W\setminus M\longrightarrow Z$ satisfying, in essence, the following condition:

$$
\begin{aligned}
 f(a,\cdot)\in&\Cc((G\cup B)\setminus M_a,Z)\cap\Oc(G\setminus M_a,Z),\quad a\in A, \\
 f(\cdot,b)\in&\Cc((D\cup A)\setminus M^b,Z)\cap\Oc(D\setminus M^b,Z),\quad b\in B,\quad \footnotemark
\end{aligned}
$$
\footnotetext{$\Cc(D',Z)$ (resp. $\Oc(D',Z)$) denotes the set of all continuous  (resp. holomorphic) mappings from a
topological  space  (resp. a complex  manifold)  $D'$ to $Z$.}

there exists an  $\hat{f}\in\Oc(\widehat{\widetilde{W}}\setminus \widehat{M},Z)$  such that for ``all" $(\zeta,\eta)\in
W\setminus M$, $\hat{f}(z,w)$ tends to $f(\zeta,\eta)$ as $(z,w)\in\widehat{\widetilde{W}}\setminus\widehat{M}$ tends,
in some sense, to $(\zeta,\eta)$. }

  \medskip

We briefly recall the very recent developments \footnote{For a more detailed history see \cite{nv4}.} around this
PROBLEM.

The case when  $M=\varnothing$ has  been  thoroughly investigated in  the work \cite{nv1,nv2} of the first author.
These articles also show that  the natural ``target spaces" $Z$ for obtaining a satisfactory answer to the above
PROBLEM are the ones which possess the {\it Hartogs extension property} \footnote{ This notion will be formulated in
Subsection \ref{HEP}  below.}.

\smallskip

The case  where $X$ and $Y$ are Riemann domains (over $\C^n$), $A\subset D,$ $B\subset G,$ and  $Z=\C$  has been
completed  in some joint-articles
  of M. Jarnicki and the second author  (see \cite{jp2,jp3,jp4,jp5}).

\smallskip

  Therefore, it is  reasonable to conjecture that a positive  solution to the PROBLEM may exist when the ``target
  space" $Z$
possesses  the Hartogs extension property. As  our first attempt  towards  an affirmative answer to the PROBLEM,  we
solve  in \cite{pn5}  the following special case:  $X=Y=\C,$ $D$ and $G$ are copies of the open unit disc in $\C$,
  $ A$ (resp. $B$) is a measurable  subset of $\partial D$ (resp. $\partial G$) of positive
  one-dimensional Lebesgue measure,
     $ Z=\C$,  and  $M$  is
 polar (resp. discrete) in fibers  over $A$  and $B.$

\smallskip

The main purpose of this article  is to  verify the above conjecture in  its full generality. Our proof is  geometric
in nature. Indeed, our method consists in  using holomorphic discs, and it  is based on the works in
\cite{jp3,pn5,nv1,nv2}. Moreover, the novelty of this new approach is that \textit{it does not use the classical
method of doubly orthogonal bases of Bergman type.} It is  worthy to note here that  most of  the previous works in
the subject of separate holomorphy make  use  of the latter method.

\smallskip

\indent{\it{\bf Acknowledgment.}}
   The paper was written while the first  author  was visiting the  Abdus Salam International
Centre
 for Theoretical Physics
in Trieste  and the Korea Institute for Advanced Study in Seoul.
 He wishes to express his gratitude to these organizations.

\section{Preliminaries and the statement of the main result} \label{Section_preliminary}

First  we  recall   some   notions developed in \cite{nv2}   such as  systems of  approach
regions for an open set in a  complex manifold, and  the corresponding plurisubharmonic measures.
These will provide the  framework for an  exact  formulation of the PROBLEM and for  our final solution.

 \subsection{Approach regions, local pluripolarity and plurisubharmonic measure}\label{Subsection_approach_regions}

\begin{definition}\label{defi_approach_region}
  Let $X$ be a complex manifold and  $D\subset X$ an open subset.
  A  {\rm system of approach regions} for $D$ is a collection
  $\mathcal{A}=\big(\mathcal{A}_{\alpha}(\zeta)\big)_{\zeta\in\overline{ D},\  \alpha\in I_{\zeta}}$
  ( $I_{\zeta}\neq\varnothing$ for all  $\zeta\in\partial D$) of open subsets of $D$
  with the following properties:
  \begin{itemize}
  \item[(i)] For all $\zeta\in D,$ the system $\big(\mathcal{A}_{\alpha}(\zeta)\big)_{ \alpha\in I_{\zeta}}$
  forms a basis of open neighborhoods of $\zeta$  (i.e., for any  open neighborhood $U$  of a point  $\zeta\in
  D,$
  there is an $ \alpha\in I_{\zeta}$ such that  $\zeta\in \mathcal{A}_{\alpha}(\zeta)\subset U$).
  \item[(ii)] For  all $\zeta\in\partial D$ and
     $\alpha\in I_{\zeta},$
  $\zeta\in \overline{\mathcal{A}_{\alpha}(\zeta)}.$
  \end{itemize}

 Moreover, $\mathcal{A}$  is
said to be {\rm canonical} if it satisfies (i) and the following  property (which is stronger than (ii)):
  \begin{itemize}
  \item[(ii')]
For every point  $\zeta\in \partial D,$
  there is  a basis of open neighborhoods  $(U_{\alpha})_{\alpha\in I_{\zeta}}$ of $\zeta$  in $X$ such that
   $ \mathcal{A}_{\alpha}(\zeta)=U_{\alpha}\cap D,$  $\alpha\in I_{\zeta}.$
  \end{itemize}
 $\mathcal{A}_{\alpha}(\zeta)$ is often
  called an {\rm approach region} at $\zeta.$
  \end{definition}

  In what follows we fix an open subset $D\subset X$ and a  system of approach regions
  $\mathcal{A}=\big(\mathcal{A}_{\alpha}(\zeta)\big)_{
  \zeta\in\overline{ D},\
   \alpha\in I_{\zeta}}$ for $D.$

 For every function $u:\ D\longrightarrow [-\infty,\infty),$ let
\begin{equation*}
 (\Alimsup u)(z):=
\sup\limits_{\alpha\in I_{z}}\limsup\limits_{ \mathcal{A}_{\alpha}(z)\ni w\to z}u(w),\quad  z\in\overline{D}.
\end{equation*}
Therefore,
\begin{equation*}
 (\Alimsup u)(z)=
 \limsup\limits_{D\ni w\to  z} u(w), \quad  \text{if}\ z\in D,
\end{equation*}
i.e.
 $(\Alimsup u)|_D$ coincides with the usual
{\it upper semicontinuous regularization}  of $u$  in case $u$ is locally bounded from above on $D$.

For a set  $A\subset \overline{D}$ put
\begin{equation*}
h_{A,D}:=\sup\left\lbrace u\ :\  u\in\mathcal{PSH}(D),\ u\leq 1\ \text{on}\ D,\
   \Alimsup u\leq 0\ \text{on}\ A    \right\rbrace,
\end{equation*}
where $\mathcal{PSH}(D)$ denotes the cone  of all functions plurisubharmonic on $D.$

A set $A\subset D$ is said to be {\it pluripolar} in $D$ if there is  $u\in \mathcal{PSH}(D)$ such that $u$ is not
identically $-\infty$ on every connected component of $D$ and $A\subset \left\lbrace z\in D:\
u(z)=-\infty\right\rbrace.$   A set  $A\subset D$ is said to be {\it locally pluripolar} in $D$ if  for any $z\in A,$
there is an open neighborhood $V\subset D$ of $z$ such that $A\cap V$ is pluripolar in $V.$ A set $A\subset D$ is said
to be {\it non-pluripolar} (resp. {\it non-locally  pluripolar}) if it is not pluripolar (resp. not locally
pluripolar). According to a classical result of Josefson and Bedford (see
\cite{jo}, \cite{be}), if $D$ is a Riemann domain over a Stein manifold, then   $A\subset D$ is   locally  pluripolar
if and only if it is pluripolar.

\begin{definition}\label{defi_relative_extremal}
 The
  {\rm relative extremal function of $A$  relative to $D$} is
 the function $ \omega(\cdot,A,D)$
 defined by
 \begin{equation*}
 \omega(z,A,D)=\omega_{\mathcal{A}}(z,A,D):= (\Alimsup h_{A,D})(z),\qquad  z\in\overline D.\quad \footnotemark
\end{equation*}
\end{definition}
\footnotetext{Observe that this function depends on the system of approach regions.}
 Note that when $A\subset D,$  Definition \ref{defi_relative_extremal}
   coincides with the classical
definition of Siciak's  relative extremal function for $z\in D$.

Next, we say that a  set  $A\subset \overline{D}$ is {\it locally pluriregular at a point $a\in \overline{A}$}   if
$\omega(a,A\cap U,D\cap U)=0$ for  all open neighborhoods $U$ of $a$, where the system of approach regions for $D\cap
U$ is given by $\mathcal{A}|_{D\cap U}:=(\mathcal{A}_\alpha(z)\cap U)_{z\in\overline{D\cap U},\; \alpha\in I_z}$.
Moreover, $ A$ is said to be {\it locally pluriregular } if it is locally pluriregular at all points $a\in A.$ It
should be noted from  Definition \ref{defi_approach_region} that  if $a\in \overline{A}\cap D$, then the property of
local pluriregularity of $A$ at $a$  does not depend on the system of approach regions $\mathcal{A},$ while the
situation is different when  $a\in \overline{A}\cap \partial D$: then the property does depend on $\mathcal{A}.$

 We denote by $A^{\ast}$ the following set
 \begin{equation*}
  (A\cap\partial D)\bigcup\left\lbrace   a\in \overline{A}\cap D:\ A\ \text{is locally pluriregular at}\ a
  \right\rbrace.
  \end{equation*}
If $A\subset D$ is non-locally pluripolar, then  a classical result of Bedford and Taylor (see \cite{be,bt}) says that
$A^{\ast}$ is locally pluriregular
 and  $A\setminus A^{\ast}$ is locally  pluripolar.
Moreover, when $A\subset D,$  $A^{\ast}$ is locally  of type $\mathcal{G}_{\delta},$ that  is, for every $a\in
A^{\ast}$ there is an open  neighborhood $U\subset D$ of $a$ such that $A^{\ast}\cap U$ is  a countable intersection
of open sets.

Now we are  in the position to introduce the following  version of a plurisubharmonic measure.
\begin{definition}\label{defi_pluri_measure}
For a set $A\subset \overline{D},$ let $\widetilde{A}=\widetilde{A}(\mathcal{A}):=\bigcup\limits_{P\in \mathcal{E}(A)}
P,$ where
\begin{equation*}
\mathcal{E}(A)=\mathcal{E}(A,\mathcal{A}):=\left\lbrace P\subset \overline{D}:\  P\ \text{is  locally pluriregular,}\
\overline{P}\subset  A^{\ast} \right\rbrace.\;\footnote{ Note that $\overline{P}\subset (A\cap\partial D)\cup (A\cap
D)^\ast$.}
\end{equation*}
 The {\rm  plurisubharmonic measure of $A$  relative to $D$} is
 the function $\widetilde{\omega}(\cdot,A,D)$
 defined by
\begin{equation*}
\widetilde{\omega}(z,A,D):=  \omega(z,\widetilde{A},D),\qquad  z\in \overline{D}.
\end{equation*}
\end{definition}

It is worthy to remark that $\widetilde{\omega}(\cdot,A,D)|_D\in\mathcal{PSH}(D)$ and $0\leq
\widetilde{\omega}(z,A,D)\leq 1,\ z\in  D.$  Obviously, if $\widetilde A\neq\varnothing$, then $\widetilde A$ is
locally pluriregular; in particular,
\begin{equation}\label{eq_defi_pluri_measure}
 \widetilde{\omega}(z,A,D)=0,\qquad z\in
\widetilde{A}.
\end{equation}
 An example in \cite{ah} shows  that, in general,
$\omega(\cdot,A,D)\not=\widetilde{\omega}(\cdot,A,D)$ on $D.$

Now  we compare the plurisubharmonic measure $ \widetilde{\omega}(\cdot,A,D)$ with Siciak's relative  extremal
function $ \omega(\cdot,A,D).$ For the moment, we only focus  on the case  where  $A\subset D.$

  If $A$ is an open subset of an arbitrary complex manifold  $D$, then it can be shown that
\begin{equation*}
\widetilde{\omega}(z,A,D)=  \omega(z,A,D),\qquad  z\in D.
\end{equation*}
If $A$  is  a (not necessarily open) non-locally pluripolar subset of an arbitrary complex manifold $D,$ then we have,  by Proposition 7.1 in \cite{nv2}, 
  \begin{equation*}
\widetilde{\omega}(z,A,D)=  \omega(z,A^{\ast},D),\qquad  z\in D.
\end{equation*}
On the other hand,  if, morever,  $D$  is  a  bounded  open subset of  $\C^n$,   then we  have (see, for  example,
  Lemma 3.5.3 in \cite{jp1}) $
\omega(z,A,D)=  \omega(z,A^{\ast},D),$ $ z\in D.$ Consequently, under  the last assumption,
\begin{equation*}
\widetilde{\omega}(z,A,D)=  \omega(z,A,D),\qquad  z\in D.
\end{equation*}
The case where $A\subset\partial D$ has been investigated  in \cite{nv2,nv3}.
 Our discussion  shows that, at least  in the case  where  $A\subset D$, the notion of  the plurisubharmonic measure
is  a  good candidate  for  generalizing  Siciak's relative  extremal function  to the manifold context in the theory
of separate holomorphy.

 For a good background of the pluripotential
theory, see the books  \cite{jp1} or  \cite{kl}. For a  more detailed
discussion on  systems of approach regions as well as their  corresponding plurisubharmonic measure, see \cite{nv1}.

\subsection{Cross, separate holomorphicity, and $\mathcal{A}$-limit.} \label{subsection_cross}

Let $X,\ Y$  be two complex manifolds,
  let $D\subset X,$ $ G\subset Y$ be two nonempty open sets, let
  $A\subset \overline{D}$   and  $B\subset \overline{G}.$
  Moreover, $D$  (resp.  $G$) is equipped with a
  system of approach regions
  $\mathcal{A}(D)=\big(\mathcal{A}_{\alpha}(\zeta)\big)_{\zeta\in\overline{D},\  \alpha\in I_{\zeta}}$
  (resp.  $\mathcal{A}(G)=\big(\mathcal{A}_{\alpha}(\eta)\big)_{\eta\in\overline{G},\  \alpha\in
  I_{\eta}}$).\;\footnote{ In fact we should have written $I_\zeta(D)$,
  resp. $I_\eta(G)$; but we skip $D$ and $G$ here to make the notions
  as simple as possible.}
 We define
a {\it $2$-fold cross} $W,$  its {\it  interior} $W^{\text{o}}$ and its {\it  regular part} $\widetilde{W}$ (with
respect to $\mathcal{A}(D)$ and $\mathcal{A}(G)$) as
\begin{eqnarray*}
W &=&\X(A,B; D,G) :=\big((D\cup A)\times B\big)\bigcup\big (A\times(B\cup G)\big),\\ W^{\text{o}}
&=&\X^{\text{o}}(A,B; D,G) := (A\times  G)\cup (D\times B),\\ \widetilde{W} &=&\widetilde{\X}(A,B;D,G) :=\X(\widetilde
A,\widetilde B;D,G).
\end{eqnarray*}
Moreover, put
\begin{eqnarray*}
\omega(z,w)&:=&\omega(z,A,D)+\omega(w,B,G),\qquad (z,w)\in D\times G,\\
\widetilde{\omega}(z,w)&:=&\widetilde{\omega}(z,A,D)+\widetilde{\omega}(w,B,G),\qquad (z,w)\in D\times G.
\end{eqnarray*}

For a $2$-fold cross $W :=\X(A,B; D,G)$ let
\begin{equation*}
\widehat{W}:=\widehat{\X}(A,B;D,G) =\left\lbrace (z,w)\in D\times G:\ \omega(z,w)  <1 \right\rbrace.
\end{equation*}
 Therefore,  we
obtain
\begin{equation*}
 \widehat{\widetilde{W}} =\widehat{\X}(\widetilde{A},\widetilde{B};D,G)
 =\left\lbrace (z,w)\in D\times G :\  \widetilde{\omega}(z,w)<1
\right\rbrace.
\end{equation*}

Let $Z$ be a complex analytic space and $M\subset W$ a subset which is  relatively  closed in fibers over $A$ and
$B.$
 We say that a mapping
$f:W^{\text{o}}\setminus M \longrightarrow Z$ is {\it separately holomorphic}
  and write $f\in\Oc_s(W^{\text{o}}\setminus M ,Z),$   if,
 for all $a\in A $ (resp.  $b\in B$)
 the mapping $f(a,\cdot)|_{G\setminus M_a}$  (resp.  $f(\cdot,b)|_{D\setminus M^b}$)  is holomorphic.

 We say that a mapping $f:\  W\setminus M \longrightarrow Z$
   is  {\it separately continuous}
and write
 $f\in \Cc_s\Big( W\setminus M,Z  \Big)$
 if,
 for all $a\in A$ (resp.  $b\in B$)
 the mapping $f(a,\cdot)|_{(G\cup B)\setminus M_a   }$  (resp.  $f(\cdot,b)|_{(D\cup A)\setminus  M^b}$)  is
 continuous.

Let  $\Omega$ be  an  open subset of $D\times  G.$ A  point $(\zeta,\eta)\in\overline{D}\times\overline{G}$ is  said
to be an {\it end-point} of $\Omega$  with  respect to
 $\mathcal{A}=\mathcal{A}(D)\times \mathcal{A}(G)$ if for any $(\alpha,\beta)\in I_{\zeta}\times I_{\eta}$ there
 exist  open neighborhoods $U$ of $\zeta$
 in $X$ and $V$ of $\eta$ in $Y$   such that
  \begin{equation*}
  \Big(U\cap \mathcal{A}_{\alpha}(\zeta)\Big)   \times \Big(V\cap \mathcal{A}_{\beta}(\eta)
  \Big)\subset \Omega.
  \end{equation*}
 The  set  of  all end-points  of $\Omega$  is  denoted  by  $\End(\Omega).$

 It follows  from  (\ref{eq_defi_pluri_measure}) that if $\widetilde{A},\widetilde{B}\not=\varnothing$,
  then $  \widetilde{W}\subset \End(
 \widehat{\widetilde{W}}).$

Let $S$ be  a relatively closed  subset of $\widehat{\widetilde{W}}$  and  let $
(\zeta,\eta)\in\End(\widehat{\widetilde{W}}\setminus  S).$    Then a mapping $f:\ \widehat{\widetilde{W}}\setminus
S\longrightarrow Z$ is said to {\it admit the $\mathcal{A}$-limit $\lambda$ at $ (\zeta,\eta),$}  and one writes
\begin{equation*}
(\Alim f)(\zeta,\eta)=\lambda,\qquad\footnote{ Note that here $\mathcal{A}=\mathcal{A}(D)\times \mathcal{A}(G).$}
\end{equation*}
 if, for all $\alpha\in I_{\zeta},\  \beta\in I_{\eta},$
\begin{equation*}
\lim\limits_{ \widehat{\widetilde{W}}\setminus S\ni (z,w)\to (\zeta,\eta),\ z\in \mathcal{A}_{\alpha}(\zeta),\  w\in
\mathcal{A}_{\beta}(\eta)}f(z,w)=\lambda.
\end{equation*}

We conclude this introduction with a notion we need in the sequel.  Let   $\mathcal{M}$ be  a topological  space. A
mapping  $f:\ \mathcal{M}\longrightarrow  Z$ is said to be {\it bounded} if there exists an open neighborhood $U$ of
$f(\mathcal{M})$ in $Z$ and a holomorphic embedding $\phi$ of $ U $ into a bounded  polydisc of $ \C^k$  such that
$\phi(U)$ is an analytic set in this polydisc.
      $f$ is said to be {\it locally bounded along} $\mathcal{N}\subset \mathcal{M}$ if
for every point $z\in \mathcal{N},$ there is an open neighborhood $U$ of $z$ (in $\mathcal{M}$) such that
 $f|_{U}:\ U \longrightarrow Z$ is bounded.
 $f$ is said to be {\it locally bounded} if  it is so for  $\mathcal{N}= \mathcal{M}.$
  It is clear that,  if $Z=\C$, then the above notions of boundedness coincide with the usual ones.

\subsection{Hartogs extension property.}\label{HEP}

We recall here  the following notion (see,  for example, Shiffman \cite{sh1} and a result by Ivashkovich \cite{iv1}).
For $0<r<1,$ the {\it Hartogs figure}, denoted by $H(r),$ is given by
 \begin{equation*}
H(r):=\left\lbrace (z_1,z_2)\in E^2: \  |z_1|<r \ \ \text{or}\ \ |z_2| >1-r \right\rbrace,
 \end{equation*}
where, in this  article,  $E$ always denotes the open  unit disc of $\C$.

\begin{definition}\label{defi_HEP}
A complex analytic space $Z$ is said to {\rm possess the Hartogs extension property} if  every  mapping
$f\in\Oc(H(r),Z)$ extends to a  mapping $\hat{f}\in\Oc(E^2,Z)$, $r\in (0,1)$.
\end{definition}

We  mention  an important characterization due to Shiffman (see \cite{sh1}).
\begin{theorem}\label{thm_Shiffman}
A complex analytic space $Z$   possesses the Hartogs extension property
    if and only if  for every subdomain $D$ of any Stein manifold $\mathcal{M},$ every mapping
    $f\in\Oc(D, Z)$ extends to a  mapping $\hat{f}\in
    \Oc(\widehat{D},Z) ,$   where $\widehat{D}$ is the envelope of
 holomorphy\footnote{ For the notion of the envelope of holomorphy, see, for example, \cite{jp1}.}
of $D.$
\end{theorem}

In the light of this result,
 the natural ``target spaces" $Z$ for obtaining
satisfactory answers to the PROBLEM are the complex analytic spaces satisfying the Hartogs extension property.

\subsection{Statement of the  main result}

Recall that a subset $S$ of a complex  manifold $\mathcal{M}$ is said to be  {\it thin} if  for    every point $x\in\mathcal{M}$  there  are a connected
 neighborhood $U=U(x)\subset\mathcal{M}$ and  a holomorphic  function $f$ on $U,$ not identically  zero,
 such that  $U\cap S\subset f^{-1}(0).$ We are now ready to state our main result.

\renewcommand{\thethmspec}{ Main Theorem}
  \begin{thmspec}
  Let $X,\ Y$  be two complex manifolds,
  let $D\subset X,$ $ G\subset Y$ be two  open sets,  and let
  $A$ (resp. $B$) be a subset of  $\overline{ D}$ (resp.
  $\overline{ G}$).  $D$  (resp.  $G$) is equipped with a
  system of approach regions
  $\big(\mathcal{A}_{\alpha}(\zeta)\big)_{\zeta\in\overline{ D},\  \alpha\in I_{\zeta}}$
  (resp.  $\big(\mathcal{A}_{\beta}(\eta)\big)_{\eta\in\overline{ G},\  \beta\in I_{\eta}}$).
  Suppose in addition that $A=A^{\ast},$ $B=B^{\ast}$
  \footnote{     It is  worthy to note that this assumption is  not so restrictive since
  we know  from  Subsection \ref{Subsection_approach_regions} that $A\setminus A^{\ast}$  and   $ B\setminus
  B^{\ast}$
  are locally pluripolar for arbitrary sets $A\subset\overline{D},$ $B\subset\overline{G}$.    }
  and  that  $\widetilde{\omega}(\cdot,A,D)<1$ on $D$ and  $\widetilde{\omega}(\cdot,B,G)<1$ on $G.$
   Let $Z$ be a complex analytic  space possessing the  Hartogs extension property.
  Let $M$ be  a relatively closed  subset of  $W$  with the following properties:
\begin{itemize}
\item[$\bullet$] $M$ is   thin  in  fibers  (resp.   locally   pluripolar in fibers) over  $A$ and  over $B;$
\item[$\bullet$] $M\cap \big((A\cap\partial D)\times B\big)= M\cap \big( A\times (B\cap\partial G)\big)=\varnothing.$
\end{itemize}
Then  there exists a relatively closed    analytic  (resp. a  relatively  closed locally  pluripolar) subset
 $\widehat{M}$ of $\widehat{\widetilde{W}},$   $\widehat{M}\cap \widetilde{W}\subset M$ \footnote{ Note that
 if $\widetilde A\cap
 D=\varnothing$ and $\widetilde B\cap G=\varnothing$, then this intersection is empty.} and
  $\widetilde{W}\setminus M\subset  \End(\widehat{\widetilde{W}}\setminus \widehat{M}),$  and
 for  every mapping    $f:\ W\setminus M\longrightarrow Z$
   satisfying the following  conditions:
   \begin{itemize}
   \item[(i)]    $f\in\Cc_s(W\setminus M,Z)\cap \Oc_s(W^{\text{o}}\setminus M,Z);$
    \item[(ii)] $f$ is locally bounded   along  $\X(A\cap\partial D,B\cap\partial G;D,G)
    \setminus M;$  \footnote{
    It follows from  Subsection \ref{subsection_cross}  that $$ \X\big (A\cap\partial D,B\cap\partial G;D,G\big
    )=
   \big( (D\cup A)\times (B\cap\partial G) \big) \bigcup
   \big((A\cap\partial D)\times ( G\cup B)\big) .$$}
    \item[(iii)]          $f|_{(A\times B)\setminus M}$ is  continuous at  all  points of
   $(A\cap\partial D)\times (B\cap \partial G),$
    \end{itemize}
     there exists  a unique  mapping
$\hat{f}\in\Oc(\widehat{\widetilde{W}}\setminus \widehat{M} ,Z)$ which  admits the $\mathcal{A}$-limit $f(\zeta,\eta)$
at  every point
  $(\zeta,\eta)\in  \widetilde{W}\setminus M  .$
\end{thmspec}

Although our main  result has been stated only  for the case  of a $2$-fold cross, it can be
also formulated   for the general case of an $N$-fold cross  with $N\geq 2$  (see also \cite{jp3, nv1,pn1}). It
remains an open question whether $\widehat{\widetilde {W}}$ is the maximal extension region of $W$ for the family of
mappings discussed in the Main Theorem  (for a special case see \cite{pn4}). Various  applications of the Main Theorem
will be given in Section  \ref{section_applications} below. It is  possible  to obtain  a generalization of  the Main
Theorem  in the  case  where $M$ is  not necessarily  closed  in $W.$ Indeed, it suffices  to   make  use of the works
\cite{jp6,jp5} and combine them with  our present  method.

Before going further we say some words  about the exposition of the paper. We only  give  the proof of the Main
Theorem  for the case where the set of singularities $M$ is locally   pluripolar in fibers. It is therefore left to
the interested reader to treat the  case where $M$ is thin  in  fibers. On the other hand, as in
any  article of  holomorphic extension, there  are always two parts: describing the method of extension and
justifying the gluing process. Since our primary aim is to make the article as compact as possible, we focus more on
the way we extend the mappings than the gluing process.
 Throughout the paper, {\bf $Z$ always denotes  a complex analytic
space possessing the Hartogs extension property.}

\section{Auxiliary results}

 First  we recall and prove  some auxiliary results. From \cite{jp3} we extract the following particular case of a general cross theorem with singularities which will be needed in the future.

\begin{theorem} \label{thm_Jarnicki_Pflug}
 Let $X=\C^n$  and  $Y=\C^m,$
  let $D\subset X,$ $ G\subset Y$ be two  bounded domains, let
  $A\subset D$  and $B\subset G$ be non-pluripolar  subsets.
  Let  $M$ be a relatively closed subset of $W$  such that $M$ is    pluripolar
  in fibers over $ A$ and  $ B.$

   Then there exists a relatively closed pluripolar  set $\widehat{M}\subset \widehat{W}$
   such that:
   \begin{itemize}
   \item[$\bullet$]
   $\widehat{M}\cap W\cap  \widetilde{W}\subset M;$
   \item[$\bullet$]
   for every  mapping  $f\in \Oc_s(W\setminus M,Z),$
 there exists a unique mapping
$\hat{f}\in\Oc(\widehat{W}       \setminus \widehat{M},Z )$ such that
 $\hat{f}=f$ on $ (W\cap   \widetilde{W})\setminus M.$
 \end{itemize}
\end{theorem}
\begin{proof} The special case   when $D$ and $G$ are pseudoconvex and  $Z=\C$  has been proved in \cite{jp3}.
However, using   a recent result in \cite{nv1}, the assumption that $D$ and $G$ are pseudoconvex  can be removed.
Now  we  treat the general case  where $Z$ is a complex analytic
space possessing the Hartogs extension property.
Applying Theorem \ref{Nguyen_thm} below  and using the hypothesis that $M$ is a relatively closed subset of $W,$
we may obtain a local extension of $f$   defined on some open neighborhood of $\widetilde{W}\setminus M.$
Finally, by applying Theorem \ref{thm_Shiffman}, the
desired conclusion of the theorem follows from   its special case   $Z=\C$ (see also \cite{az}).
\end{proof}

We also  need  the following version  of Theorem  \ref{thm_Jarnicki_Pflug}  when  $M$ is  not necessarily  closed in
$W$

\begin{theorem} \label{thm_Jarnicki_Pflug_new_version}
 Let $X=\C^n$  and  $Y=\C^m,$
  let $D_0\subset D\subset X,$ $G_0\subset G\subset Y$ be four bounded domains,  and let
  $A\subset D_0$  and $B\subset G_0$ be non-pluripolar  subsets.
  Let  $M$ be a  subset of $W:=\X(A,B;D,G)$  such that $M$ is  relatively closed  pluripolar
  in fibers over $ A$ and  $ B.$
   Then there exist:
\begin{itemize}
   \item[$\bullet$]
   pluripolar sets  $P\subset A,$  $Q\subset B$    such that   the set
$A_0:=A\setminus P,$ $B_0:=B\setminus  Q$  are locally pluriregular,
   \item[$\bullet$]
a relatively closed pluripolar  set $\widehat{M}\subset \widehat{W}$
\end{itemize}
   such that:
   \begin{itemize}
   \item[$\bullet$]
   $\widehat{M}\cap \X(A_0,B_0;D,G)\subset M;$
   \item[$\bullet$]
   for every  mapping  $f\in \Oc_s(W\setminus M,Z)\cap \Oc(D_0     \times G_0,Z
   ),$
 there exists a unique mapping
$\hat{f}\in\Oc(\widehat{W}       \setminus \widehat{M},Z )$ such that
 $\hat{f}=f$ on $ D_0\times  G_0.$
 \end{itemize}
\end{theorem}
\begin{proof} The special case   when $D_0,\ D$ and $G_0,\ G$ are pseudoconvex and  $Z=\C$  has been proved in Theorem
3.4  of \cite{jp6}.
However, using   a recent result in \cite{nv1}, the assumption of pseudoconvexity  can be removed. Finally, by
applying Theorem \ref{thm_Shiffman}, the desired conclusion of the theorem  follows from   its special case   $Z=\C.$
\end{proof}

The next result was  proved by the first author  in \cite{nv2}.
\begin{theorem}\label{Nguyen_thm}
We keep the hypotheses and notation of the Main Theorem. Suppose in addition that $M=\varnothing.$ Then the conclusion
of the Main Theorem  holds for $\widehat{M}=\varnothing.$
\end{theorem}

The following result will play an important role in the sequel.

\begin{theorem}[\cite{Chi}]\label{Chirka_thm}
Let $D\subset\C^n$ be a domain and let $\widehat{ D}$ be the envelope of holomorphy of $D$. Assume that $S$ is  a
relatively closed pluripolar subset of $D$. Then there exists a relatively closed pluripolar subset $\widehat{ S}$ of
$\widehat{ D}$ such that $\widehat{ S}\cap D\subset S$ and $\widehat{ D}\setminus\widehat{S}$ is the envelope of
holomorphy of $D\setminus S$.
\end{theorem}

In this article, let $\mes$ denote the Lebesgue measure on the unit circle $\partial E.$
 Recall here the system of  angular (or  Stolz)    approach regions for $E$ (see, for example,
 \cite{nv2}).
 Put
\begin{equation*}
\mathcal{A}_{\alpha}(\zeta):=
 \left\lbrace          t\in E:\ \left\vert
 \arg\left(\frac{\zeta-t}{\zeta}\right)
 \right\vert<\alpha\right\rbrace   ,\qquad  \zeta\in\partial E,\  0<\alpha<\frac{\pi}{2},
\end{equation*}
where  $\arg:\ \C\longrightarrow (-\pi,\pi]$ is  as usual the argument function. $
\mathcal{A}=\left(\mathcal{A}_{\alpha}(\zeta)\right)_{\zeta\in\partial E,\
 0<\alpha<\frac{\pi}{2}}$ is   referred to as   {\it   the system of  angular (or  Stolz)
    approach regions for $E.$}
In this  context  $\Alim$ is also called {\it angular limit}.

 For $z\in\C^n$ and $r>0,$ let
$\Delta^n_z(r)$ denote the open polydisc centered  at $z$ with radius $r.$  When $n=1$,  we  will write  for short
$\Delta_z(r)$ instead of $\Delta^{1}_z(r).$

Fix $A\subset\partial E$. For $a\in\partial E$ and
$0<\rho,\epsilon<1,$ let
\begin{equation*}
\Delta_{a}(\rho,\epsilon):=\Delta_{a}(\rho,\epsilon;A):=\left\lbrace  z\in\Delta_{a}(\rho) \cap E:\ \omega\big(z,
A\cap \Delta_{a}(\rho),  \Delta_{a}(\rho) \cap E\big)<\epsilon    \right\rbrace.
\end{equation*}
 It is  worthy  to remark that  if  $a$ is  a density
point of $A$, then $\Delta_{a}(\rho,\epsilon)\not=\varnothing.$

 The following result will be very useful.
\begin{proposition}\label{Imomkulov_thm}
Let  $D=G=E$ and  let $A\subset\partial D $   be a measurable  subset such that $\mes(A)>0,$ and let  $B\subset G$ be
an open set. Moreover, we assume that any point of $A$ is a density point of $A$. Consider the cross
$W:=\X(A,B;D,G).$ Let $M$ be a relatively closed subset of $W$ such that  $M_a$ is   polar (resp. discrete)  in $G,$
$M_a\cap B=\varnothing$ for all $a\in A$ and
 $M^b=\varnothing$    for all $b\in B.$   Let $B_0$  be  an open subset of $B$ such that $B_0$  intersects all connected components of $B.$ Let $S$ be a
relatively closed pluripolar  subset (resp.  analytic subset)   of $D\times B$ such that for  every $(a,b)\in A\times
B_0,$  there  exist $0<\rho,\epsilon<1$ and  an open neighborhood $V\subset G$  of  $b$  such that
$\big(\Delta_{a}(\rho,\epsilon)\times V\big)\cap S=\varnothing.$ Then there exists a relatively closed pluripolar
subset (resp. an analytic subset)  $ T$ of $\widehat{W}$ with  $T\cap( D\times B)=\varnothing$ and with the following
property:  Let $f:\ W\setminus (M\cup S) \longrightarrow \C$
 be a locally bounded
function  such that
\begin{itemize}
\item[$\bullet$]  for all $a\in A,$
 $f(a,\cdot)|_{G\setminus M_a}$   is
holomorphic;
 \item[$\bullet$] for all $b\in B$
$f(\cdot,b)|_{D\setminus S^b}$ is holomorphic, and for all $b\in B_0$ the function $f(\cdot,b)$ admits the
angular limit $f(a,b)$ at every point $a\in A.$
\end{itemize}
  Then
there is a unique function $\hat{f}\in\Oc(\widehat{W}\setminus  T,\C)$ which
 extends $f|_{(D\times  B)\setminus S}.$

 Moreover,   if $M=S=\varnothing,$ then $T=\varnothing.$
\end{proposition}
\begin{proof}
Using the hypotheses that  $f|_{\Delta_{a}(\rho,\epsilon)\times V}$ is holomorphic and that $B_0$ intersects all
connected components of $B,$ we can adapt  the argument given in the proof of Proposition 4.1  in \cite{pn5} so that
the latter proposition is still  true  in our context. The remaining  of the proof follows along the same lines  as
those  given in \cite{pn5} making the obviously necessary changes.
\end{proof}
\begin{remark}\label{remark_Imomkulov_thm}
The previous  proposition  still holds if we replace the domain $D:=E$  by the open set  $D:=\{  z\in E:\
\omega(z,A,E)<\epsilon\}$  for some $0<\epsilon<1.$
\end{remark}

The first main purpose  of this  section is  to establish the  following  higher dimensional version of Theorem
\ref{Imomkulov_thm}.

\begin{theorem} \label{new_Imomkulov_thm}
Let $A$ be  a measurable subset of $\partial E$ with $\mes(A)>0$ and let $r>1.$ Let $M$ be  a relatively closed
subset of $A\times \Delta^n_0(r) $ such that $M\cap(A\times \overline{E}^{n})=\varnothing$ and that
$M_a:=\{w\in\Delta^n_0(r):\ (a,w)\in M\}$ is  pluripolar  for all $a\in A.$ Then there exists a relatively closed
pluripolar  subset $\widehat{M}$ of $  \widehat{\X}(A,E^n;E,\Delta^n_0(r))    $ with $\widehat{M}\cap
E^{n+1}=\varnothing$ and   with the following additional property:

Let  $f:\ \X(A,E^n;E,\Delta^n_0(r))\setminus M \rightarrow \C$ be a locally bounded function  such that
\begin{itemize}
\item[ $\bullet$] for all $a\in A,$
  $f(a,\cdot)|_{\Delta^n_0(r)\setminus M_a}$  is  holomorphic;
\item[ $\bullet$]  for all $ w\in E^n,$ the function $f(\cdot,w)|_E$ is  holomorphic and  admits the angular limit
    $f(a,w)$  at all points $ a\in A.$
\end{itemize}
 Then there is a  unique  function  $\hat{f}\in \Oc\big(\widehat{\X}(A,E^n; E, \Delta^n_0(r))\setminus
 \widehat{M},\C\big)$
 which extends $f|_{E^{n+1}}.$
\end{theorem}

The second main purpose   is  to prove  the following  generalization  of Theorem \ref{new_Imomkulov_thm}, where $D$
need  not to be  a   disc.

  \begin{theorem}\label{local_thm_1}
  Let $X=\C^m$ and  $ Y=\C^n,$
  let $D\subset X$ be a bounded open set   and  $ G=\Delta^n_0(r)$ for some  $1<r<\infty,$  let
  $A$  be a subset of  $\overline{ D}$ with $A=A^\ast$, $\widetilde{\omega}(\cdot,A,D)<1$  on $D,$ and let $B=E^n.$ Moreover, $D$   is equipped with a
  system of approach regions
  $\big(\mathcal{A}_{\alpha}(\zeta)\big)_{\zeta\in\overline{ D},\  \alpha\in I_{\zeta}}$
  and $G$  is  equipped with the canonical  system        of approach
  regions.
  Let $M$ be  a closed  subset of  $W:=\X(A,B;D,G)$  with the following properties:
\begin{itemize}
\item[$\bullet$] $M$ is     locally   pluripolar in fibers over  $A$ and  over $B;$
 \item[$\bullet$]    $M\cap
    (A\times B)=\varnothing,$ $M\cap (D\times B)=\varnothing.$
\end{itemize}
Then  there exists  a  relatively  closed locally pluripolar subset
 $\widehat{M}$ of $\widehat{W}$ with  $\widehat{M}\cap(D\times B)=\varnothing,$
 $W\cap \widetilde{W}\cap \widehat{M}\subset M$  and $(W\cap \widetilde{W})\setminus M
 \subset \End(\widehat{W}\setminus \widehat{M})$  such that:

   for  every function   $f:\ W\setminus M\longrightarrow \C$
   satisfying the following  conditions:
   \begin{itemize}
   \item[$\bullet$]    $f\in\Cc_s(W\setminus M,\C)\cap \Oc_s(W^{\text{o}}\setminus M,\C);$
    \item[$\bullet$] $f$ is locally bounded   along  $ \big ((A\cap\partial D)\times
G\big)
    \setminus M;$
    \end{itemize}
     there exists  a unique  function
$\hat{f}\in\Oc(\widehat{\widetilde{W}}\setminus \widehat{M} ,\C)$ which admits the $\mathcal{A}$-limit $f(\zeta,\eta)$
at  every point
  $(\zeta,\eta)\in  (W\cap   \widetilde{W})\setminus M  .$
\end{theorem}

In order to prove  Theorem \ref{new_Imomkulov_thm} and  \ref{local_thm_1}, our  strategy is  as  follows. First,
observe  that Theorem    \ref{new_Imomkulov_thm}      for  the  case $n=1$ follows  from Theorem \ref{Imomkulov_thm}.
Next, we will show that  Theorem    \ref{new_Imomkulov_thm}   for  a given $n$    implies  Theorem \ref{local_thm_1}
for this $n .$ Finally, it suffices  to show that Theorem \ref{local_thm_1} for $n=1$  implies,  in turn,  Theorem
\ref{new_Imomkulov_thm}      for  arbitrary  $n.$

\smallskip

To   make the  above  strategy  to work we  will  rely on the approach using holomorphic  discs  as  it was  done  in
\cite{nv2}. For a bounded mapping $\phi\in\Oc(E,\C^n)$ and $\zeta\in
\partial E,$ $f(\zeta)$ denotes the angular limit value of $f$ at
$\zeta$ if it exists. A classical theorem of Fatou says that $\mes\left(\{\zeta\in\partial E:\ \exists
f(\zeta)\}\right)=2\pi.$

\begin{theorem}\label{Poletsky}
Let $D$ be a bounded  open set in $\C^n,$  $A\subset \overline{D},$ $z_0\in D$  and $\epsilon>0.$ Let  $\mathcal{A}$ be
a system of approach regions for $D.$ Suppose in addition that  $A$ is locally pluriregular (relative to
$\mathcal{A}$) and that $\omega(\cdot,A,D)<1$ on $D.$ Then  there exist a bounded  mapping $ \phi\in\Oc(E,\C^n)$ and a measurable subset $\Gamma_0\subset
\partial E$ with the following  properties:
\begin{itemize}
\item[1)] Any point of $\Gamma_0$ is a density point of $\Gamma_0,$  $\phi(0)=z_0,$
    $\phi(E)\subset  \overline{D},$ $\Gamma_0 \subset \left\lbrace \zeta\in\partial E:\ \phi(\zeta)\in
    \overline{A} \right\rbrace,$  and
\begin{equation*}
1-\frac{1}{2\pi}\cdot\mes (\Gamma_0 )<\omega(z_0,A,D)+\epsilon.
\end{equation*}
\item[2)] Let  $f\in\Cc(D\cup \overline{A},\C)\cap \Oc(D,\C)$  be such that $f(D)$
 is bounded.  Then there exist a bounded function $g\in\Oc(E,\C)$ such that
  $g=f\circ \phi$ in  a neighborhood of $0\in E$  and \footnote{ Note  here that  by Part 1),  $(f\circ
  \phi)(\zeta)$ exists
   for all $\zeta\in \Gamma_0.$}
 $g(\zeta)=(f\circ \phi)(\zeta)$  for all $\zeta\in \Gamma_0.$ Moreover, $g|_{\Gamma_0}\in\Cc(\Gamma_0,\C).$
 \end{itemize}
\end{theorem}

This  theorem   which  has been  proved  in Theorem 3.8 of \cite{nv2}   motivates the following
\begin{definition}
\label{candidate} We keep the hypotheses and notation of Theorem \ref{Poletsky}. Then every pair $(\phi,\Gamma_0)$
satisfying the conclusions 1)--2) of this theorem
 is said to be an {\rm $\epsilon$-candidate for the triplet $(z_0,A,D).$}
\end{definition}

Theorem \ref{Poletsky}  says that  there always exist $\epsilon$-candidates for all triplets $(z,A,D).$

The following result reduces the Main Theorem to local situations.
\begin{proposition}\label{Nguyen_prop}
We keep  the hypotheses  and notation of the  Main Theorem. 1) Suppose  in addition that  the following property
holds:

  Let  $A_0$ (resp. $B_0$) be  a subset of
 $\overline{ D}$ (resp. $\overline{ G}$) such that $A_0$ and  $B_0$ are locally pluriregular
    and that
$\overline{A}_0\subset A$ and $\overline{B}_0\subset B$  and that $\overline{A}_0, \overline{B}_0     $ are compact.
Then there exists a relatively closed pluripolar subset $\widehat{M}$ of $\widehat{\X}(A_0,B_0;D,G)$  such that
$\X(A_0,B_0;D,G)\setminus M\subset \End(\widehat{\X}(A_0,B_0;D,G)\setminus \widehat{M})$ and that
   for every  mapping  $f:\ W\longrightarrow Z$ which   satisfies    conditions (i)--(iii) of the  Main Theorem,
there exists
   a  mapping
$\hat{f}$ defined and holomorphic on
 $\widehat{\X}(A_0,B_0;D,G)\setminus\widehat{M}$ which
  admits  $\mathcal{A}$-limit $f(\zeta,\eta)$ at all points
  $(\zeta,\eta)\in   \X(A_0,B_0;D,G)\setminus M.$

  Then the conclusion of the  Main Theorem  holds.
 \\
  2) Suppose  that  the property of Part  1) is  satisfied for all $B_0\in \mathcal{F},$
  where  $\mathcal{F}\subset\mathcal{E}(B)$  such that $\widetilde{B}=\bigcup\limits_{P\in \mathcal{F}}P.$
  Then the conclusion of the  Main Theorem  holds. Here  we recall that
  \begin{equation*}
\mathcal{E}(B)=\mathcal{E}(B,\mathcal{A}):=\left\lbrace P\subset \overline{G}:\  P\ \text{is  locally pluriregular,}\
\overline{P}\subset  B \right\rbrace.
\end{equation*}
  \end{proposition}
This result  permits to pass from the relative extremal functions $\omega(\cdot,A_0,D),$ where $A_0\in\mathcal{E}(A),$
to the plurisubharmonic measure $\widetilde{\omega}(\cdot,A,D).$
\begin{proof}
The  case  where $M=\varnothing$  is  treated by the first author in \cite{nv2} where he starts from Theorem 8.2
therein in order to prove the main theorem in Section 9 of that article. This  method also works in the present
context making the obviously necessary changes.
\end{proof}

In the light of Part 2) of  Proposition  \ref{Nguyen_prop},  Theorem  \ref{local_thm_1} is  reduced  to prove the
following  result.

\begin{theorem}\label{local_thm_2}
  Let $X=\C^m$ and  $ Y=\C^n,$
  let $D\subset X$ be a bounded open set   and  $ G=\Delta^n_0(r)$ for some  $r>1,$  let
  $A$  be a subset of  $\overline{ D}$  such that $A=A^{\ast},$  $\widetilde{\omega}(\cdot,A,D)<1$ on $D,$ and let $B=E^n.$ Moreover, $D$   is equipped with a
  system of approach regions
  $\big(\mathcal{A}_{\alpha}(\zeta)\big)_{\zeta\in\overline{ D},\  \alpha\in I_{\zeta}}$
  and $G$  is  equipped with the canonical  system of approach
  regions.   Let  $A_0$  be  a subset of
 $\overline{ D}$  such that $A_0$ is locally pluriregular
    and that
$\overline{A}_0\subset A.$  Put  $W_0:=\X(A_0,B;D,G).$
  Let $M$ be  a relatively closed  subset of  $W$  with the following properties:
\begin{itemize}
\item[$\bullet$] $M$ is     locally   pluripolar in fibers over  $A;$ \item[$\bullet$]    $M\cap \big((A\cup
    D)\times B\big)=\varnothing.$
\end{itemize}
 Then  there exists  a  relatively  closed locally pluripolar subset
 $\widehat{M}$ of $\widehat{W}_0$
 with  $\widehat{M}\cap(D\times B)=\varnothing$
 and   $(W\cap   W_0)\setminus M\subset  \End(\widehat{W}_0\setminus \widehat{M}) $ such that:

   for  every function   $f:\ W\setminus M\longrightarrow \C$
   satisfying the following  conditions:
   \begin{itemize}
   \item[$\bullet$]    $f\in\Cc_s(W\setminus M,\C)\cap \Oc_s(W^{\text{o}}\setminus M,\C);$
    \item[$\bullet$] $f$ is locally bounded   along  $ \big ((A\cap\partial D)\times
G\big)
    \setminus M;$
    \end{itemize}
     there exists  a unique  function
$\hat{f}\in\Oc(\widehat{W}_0\setminus \widehat{M} ,\C)$ which admits the $\mathcal{A}$-limit $f(\zeta,\eta)$ at  every
point
  $(\zeta,\eta)\in  (W\cap   W_0)\setminus M  .$
\end{theorem}

\begin{proof}
First  we  will  find a subset  $\mathcal{M}\subset \widehat{W}_0$ such that
\begin{itemize}
\item[$\bullet$] $\mathcal{M}\cap (D\times B)=\varnothing;$ \item[$\bullet$] for all $z\in D,$  the vertical
    fibers $\mathcal{M}_z:=\{ w\in G:\ (z,w)\in\mathcal{M}\} $    are relatively closed pluripolar in
    $(\widehat{W}_0)_z:=\{ w\in G:\ (z,w)\in \widehat{W}_0  \} $; \item[$\bullet$]  $f|_{D\times B}$ extends to
    $\hat{f}$ which is  well-defined on $ \widehat{W}_0\setminus \mathcal{M}$ and which satisfies
    $\hat{f}(z,\cdot)\in\Oc\big(( \widehat{W}_0\setminus \mathcal{M})_z,\C\big)$ for all $z\in D,$ where
$$
( \widehat{W}_0\setminus \mathcal{M})_z:=\left\lbrace w\in G:\ (z,w)\in  \widehat{W}_0\setminus \mathcal{M}
\right\rbrace.
$$
 \end{itemize}
To this end fix a $z_0\in D$. We want to construct  the vertical fiber  $\mathcal{M}_{z_0}.$  Take  an arbitrary $\epsilon>0$  such that
 \begin{equation*}
 \omega(z_0,A_0,D)+\epsilon<1.
 \end{equation*}
   By Theorem \ref{Poletsky} and Definition \ref{candidate}, there is
 an $\epsilon$-candidate $(\phi,\Gamma)$    for $(z_0,A_0,D).$
 By shrinking $\Gamma$  using Lusin's theorem, we may assume  without loss of generality
 that  $\phi|_{\Gamma}$ is  continuous.
 Moreover, using the hypotheses on $M$  and  on $f,$ we see that the function  $f_{\phi},$  defined by
 \begin{equation}\label{eq2_local_thm_1_Step1}
 f_{\phi}(t,w):=f(\phi(t),w),\qquad  (t,w)\in \X\left(\Gamma,B;E,G\right)\setminus M_{\phi},
 \end{equation}
satisfies the hypotheses of Theorem \ref{new_Imomkulov_thm}, where
$$M_{\phi}:=\{(t,w)\in \Gamma\times G:\  (\phi(t),w)\in M\}.$$
By this theorem, let $\widehat{M}_{\phi}$ be the relatively closed pluripolar subset of $
\widehat{\X}\left(\Gamma,B;E,G\right)      $ with   $\widehat{M}_{\phi}\cap (E\times B)=\varnothing$  and let
$\hat{f}_{\phi}\in\Oc\big(      \widehat{\X}\left(\Gamma,B;E,G\right)     \setminus \widehat{M}_{\phi},\C\big) $ be
such that
  \begin{equation}\label{eq3_local_thm_1_Step1}
   (\Alim \hat{f}_{\phi})(t,w)=f_{\phi}(t,w),\qquad  (t,w)\in
  \X^{\text{o}}\left(\Gamma, B;E,G\right)\setminus  M_{\phi}.
  \end{equation}

Using the   above  discussion we will define $\mathcal{M}_{z_0}$ and    the desired extension function
$\hat{f}(z_0,\cdot)$ on $( \widehat{W}_0\setminus \mathcal{M})_{z_0}$ as follows:

fix a point $(z_0,w_0)\in\widehat W$ and an $\epsilon >0$ such that
\begin{equation*}
 \omega(z_0,A_0,D)+\omega(w_0,B,G)+\epsilon<1
 \end{equation*}
 and  there exists an $\epsilon$-candidate
$(\phi,\Gamma)$    for $(z_0,A,D)$  with $(0,w_0)\in\widehat{\X}\left(\Gamma,B;E,G\right)\setminus
 \widehat{M}_{\phi}.$
 Then  the value of $\hat{f}$  at $(z_0,w_0)$ is, by our
definition, given as
\begin{equation}\label{eq_local_thm_1_Step1_formula_hatf}
 \hat{f}(z_0,w_0):=\hat{f}_{\phi}(0,w_0) ,
\end{equation}
where $\hat{f}_{\phi}$ is  defined  in (\ref{eq2_local_thm_1_Step1})--(\ref{eq3_local_thm_1_Step1}). On the other
hand, put
\begin{equation*}
\mathcal{M}_{z_0}:=\left\lbrace w\in G:\  \forall \phi\ \text{as above},\ (0,w)\in\widehat{M}_{\phi}\right\rbrace.
\end{equation*}
 Using  Lemma 4.5  in \cite{nv2} it
can be checked
  that $\hat{f}(z_0,\cdot)$ is well-defined  on $ (\widehat{W}_0)_{z_0}\setminus \mathcal{M}_{z_0}.$
 Moreover, it is  easy to see that  $\mathcal{M}\cap (D\times B)=\varnothing $
and that all vertical  fibers $\mathcal{M}_z$ with $z\in D$   are relatively closed pluripolar.  This completes the
construction of  $\mathcal{M}\subset \widehat{W}_0$ and  of $\hat{f}$  on $  \widehat{W}_0\setminus \mathcal{M}.$

For all $0<\delta<\frac{1}{2}$ let
\begin{equation}\label{eq_Step1_Adelta}
 A_{\delta}:=\left\lbrace z\in D:\
 \omega(z,A_0,D)<\delta\right\rbrace\ \text{and}\
  G_{\delta}:=\left\lbrace w\in G:\  \omega(w,B,G)<1-\delta\right\rbrace.
\end{equation}
 We are able to define a new function $\tilde{f}_{\delta}$ on $\X\left(A_{\delta},B;D,G_{\delta}
 \right)\setminus \mathcal{M}$ as follows
\begin{equation}\label{eq_tildefdelta}
 \tilde{f}_{\delta}(z,w):=
\begin{cases}
 \hat{f}(z,w)
  & \qquad (z,w)\in  (A_{\delta}\times G_{\delta})\setminus \mathcal{M}, \\
  f(z,w) &   \qquad (z,w)\in D\times B       .
\end{cases}
\end{equation}
  Using the hypotheses on $f$ and the previous  paragraph, we see  that
  $$ \tilde{f}_{\delta}\in \Oc_s\Big(
\X\left(A_{\delta}, B;D,G_{\delta}
 \right)\setminus \mathcal{M},\C\Big).$$
Observe that $A_{\delta}$ is  an open set in $D,$ all vertical  fibers $\mathcal{M}_z$ with $  z\in D$  are relatively
closed pluripolar  and  all horizontal  fibers $\mathcal{M}^w$ with $ w\in B$ are empty, and
$\tilde{f}_{\delta}|_{D\times B}$ is  holomorphic.    Consequently,
 $\tilde{f}_{\delta}$ satisfies the hypotheses  of Theorem
 \ref{thm_Jarnicki_Pflug_new_version} for $D_0:=D,$  $G_0:=B,$  $A:=A_{\delta},$ $B:=B,$ $D:=D,$ $G:=G_{\delta},$
 where the left sides of the above assignments are the notation of Theorem  \ref{thm_Jarnicki_Pflug_new_version}.
 Applying  this theorem  yields a  relatively closed  pluripolar   subset $\widehat{M}_{\delta}$  of
$\widehat{\X}\left(A_{\delta}, B;D,G_{\delta}
 \right)$ with   $\widehat{M}_{\delta}\cap (D\times B)=\varnothing $ and a  function  $\hat{f}_{\delta}\in
 \Oc\Big(\widehat{\X}\left(A_{\delta}, B;D,G_{\delta}
 \right)\setminus \widehat{M}_{\delta}  ,\C \Big)$
 such that
 \begin{equation*}
 \hat{f}_{\delta}(z,w)= \tilde{f}_{\delta}(z,w),\qquad (z,w)\in \X\left(A_{\delta}, B;D,G_{\delta}
 \right)\setminus \mathcal{M}.
\end{equation*}
 This, combined with (\ref{eq_tildefdelta}), implies that
 $\hat{f}$ (given  in (\ref{eq_local_thm_1_Step1_formula_hatf}))
 extends holomorphically to $(A_{\delta}\times G_{\delta})\setminus \widehat{M}_{\delta}.$
Note that $\widehat M_\delta$ may be taken as singular with respect to all these extended functions. On the
other hand, it follows from  (\ref{eq_Step1_Adelta}) that
 \begin{equation*}
  \widehat{W}_0=\widehat{\X}\left(A_0,B;D,G
 \right)=\bigcup\limits_{0<\delta<1}A_{\delta}\times G_{\delta}.
  \end{equation*}
  Now  fix a sequence  $(\delta_k)_{k=1}^{\infty}$ such that  $0<\delta_{k}<1$ and $\delta_k\searrow 0^{+}.$
  Then, using the last  equality, we may glue   $(\widehat{M}_{\delta_k})_{k=1}^{\infty}$  together in order to
  obtain  a  relatively closed  pluripolar   subset $\widehat{M}$  of  $\widehat{W}_0$
and an extension function $\hat{f}\in\Oc(  \widehat{W}_0\setminus \widehat{M},\C)$ with the desired  properties of the
theorem.
\end{proof}

Prior to the proof of Theorem  \ref{new_Imomkulov_thm} for all $n$  we make some preparation. Under the hypotheses and
notation  of Theorem    \ref{new_Imomkulov_thm}  we  establish  the  following

\begin{proposition} \label{prop_local_extension}
Let $A$ be  a measurable subset of $\partial E$ with $\mes(A)>0.$ Then,
  for every  density point  $a_0\in A$   and every $r^{'}\in (1,r),$ there  exist $0<\rho=\rho_{r^{'}},
\epsilon=\epsilon_{r'}<1$  and   a   relatively closed  pluripolar set $T=T_{r^{'}}\subset
\Delta_{a_0}(\rho,\epsilon)\times  \Delta_0^n(r^{'})$ with $T\cap (\Delta_{a_0}(\rho,\epsilon)\times E^n)=\varnothing$
 such that every function $f$ satisfying   the hypotheses of Theorem
 \ref{new_Imomkulov_thm} extends holomorphically to  $\big(\Delta_{a_0}(\rho,\epsilon)\times  \Delta_0^n(r^{'})\big)
\setminus T.$
\end{proposition}

In other words, this proposition  says that some  local  extensions   are  possible.
\begin{proof}
Fix a point $a_0$ of $ A$ and let $r'_0$ be the supremum of all $r'\in(0,r)$ such that $\rho_{r'},$  $\epsilon_{r'}$,
and $T_{r'}$ exist with the above properties. Note that $1\leq r'_0\leq r$. It suffices to show that $r'_0=r$.

Suppose that $r'_0<r$. Fix $r'_0<r''<r$ and choose $r'\in(0,r'_0)$ such that $\root{n}\of{{r'}^{n-1}r''}>r'_0$. Let
$\rho:=\rho_{r'}$, $\epsilon:= \epsilon_{r'}$, and $S:=T_{r'}$.

Write $w=(w',w_n)\in\C^n=\C^{n-1}\times\C$. Let $C$ denote the set of all $(a,b')\in(A\cap\Delta_{a_0}(\rho))\times\Delta_0^{n-1}(r')$ such that the fiber $M_{(a,b',\cdot)}$ is polar. 
\begin{equation*}
F:=\left\lbrace  w^{'}\in \Delta_0^{n-1}(r'):\ \mes\big(\{  z\in A: (z,w^{'})\not\in C\}\big)  >0
\right\rbrace,
\end{equation*}
where  $\mes$  denotes the linear measure  on $\partial E.$ For every  $w^{'}\in  \Delta_0^{n-1}(r')\setminus F,$  we
have  that $\mes(A\setminus  C_{w^{'}})=0,$ where $ C_{w^{'}}:=\{z\in A:\  (z,w^{'})\in C\}.$ Applying  Remark
\ref{remark_Imomkulov_thm}
 to  the  function  $f(\cdot,w^{'},\cdot)$  restricted  to the cross
\begin{equation*}
Y_{w^{'}}:=\X\big(C_{w^{'}},\Delta_0(r');\Delta_{a_0}(\rho,\epsilon),\Delta_0(r)\big)\setminus (M\cup S)
\end{equation*}
 we conclude that there exists a closed pluripolar
set $T_{w'}\subset\widehat{Y_{w'}}$  such that $T_{w^{'}}\cap \Delta_{a_0}(\rho,\epsilon)\times
\Delta_0(r^{'})\subset  S$ and every  function $f$ satisfying the hypotheses of Theorem \ref{new_Imomkulov_thm}
extends holomorphically to a function $\hat{f}_{w^{'}}$ defined on
 $\widehat{ Y}_{w^{'}}\setminus T_{w^{'}}$.

Next, we will argue  as in the  proof of Theorem \ref{local_thm_2}. More precisely, for all $0<\delta<\frac{1}{2}$
let
\begin{eqnarray*}
 A_{\delta}&:=&\left\lbrace (z,w^{'})\in \Delta_{a_0}(\rho,\epsilon)\times \Delta_0^{n-1}(r^{'}):\
 \omega\big((z,w^{'}),C,  \Delta_{a_0}(\rho,\epsilon)\times \Delta_0^{n-1}(r^{'})  \big )<\delta\right\rbrace,\\
  G_{\delta}&:=&\left\lbrace w_n\in \Delta_0(r):\  \omega(w_n,\Delta_0(r^{'}),\Delta_0(r))<1-\delta\right\rbrace.
\end{eqnarray*}
We define  a   set $\mathcal{M}$ as  follows:
\begin{equation*}
 \mathcal{M}(z,w^{'}):=
\begin{cases}
  (T_{w^{'}})_z:=\{w_n\in \Delta_0(r):\  (z,w_n)\in T_{w^{'}}\}
  & \qquad w^{'}\in   \Delta_0^{n-1}(r^{'})\setminus F, \\
        \Delta_0(r)   &   \qquad w^{'}\in F       .
\end{cases}
\end{equation*}
Moreover,  we define a new function $\tilde{f}_{\delta}$ on
$$\X\left(A_{\delta},\Delta_0(r^{'}); \Delta_{a_0}(\rho,\epsilon)\times \Delta_0^{n-1}(r^{'})     ,G_{\delta}
 \right)\setminus \mathcal{M}$$  as follows
\begin{equation*}
 \tilde{f}_{\delta}(z,w):=
\begin{cases}
 \hat{f}(z,w)
  & \qquad (z,w)\in  (A_{\delta}\times G_{\delta})\setminus \mathcal{M}, \\
  f(z,w) &   \qquad (z,w)\in  \big( \Delta_{a_0}(\rho,\epsilon)\times \Delta_0^{n-1}(r^{'})  \times
  \Delta_0(r^{'})\big)\setminus  \mathcal{M} .
\end{cases}
\end{equation*}
  Using the hypotheses on $f$ and the previous  paragraph, we see  that
  \begin{equation*}
 \tilde{f}_{\delta}\in \Oc_s\Big(
\X\left(  A_{\delta},\Delta_0(r^{'}); \Delta_{a_0}(\rho,\epsilon)\times \Delta_0^{n-1}(r^{'})     ,G_{\delta}
 \right)\setminus \mathcal{M},\C\Big).
\end{equation*}
 Arguing as in the proof of Theorem  \ref{local_thm_2} we can show that there  exists a
relatively closed  pluripolar  subset $T_n$ of $\widehat{Y_n}$  such that  every function $f$ as in the hypothesis
extends holomorphically to a  function $f_n$  defined on   $\widehat{Y_n}\setminus  T_n,$  where
\begin{equation*}
Y_n:=\X\big(C,\Delta_0(r');\Delta_{a_0}(\rho,\epsilon)\times \Delta_0^{n-1}(r^{'}),\Delta_0(r)\big).
\end{equation*}
 In order  to  calculate  $\widehat{Y_n}$ we need the following
 lemma.
\begin{lemma}
For  $(z,w')\in\Delta_{a_0}(\rho,\epsilon)\times\Delta_0^{n-1}(r'),$
\begin{multline*}
\omega\big((z,w'), C,\Delta_{a_0}(\rho,\epsilon)\times\Delta_0^{n-1}(r')\big)=\\ \max\{ \frac{1}{\epsilon}\cdot
\omega\big(z, A\cap\Delta_{a_0}(\rho),\Delta_{a_0}(\rho)\big),\; \omega(w', \Delta_0^{n-1}(r'),\Delta_0^{n-1}(r')\}.
\end{multline*}
\end{lemma}
\begin{proof}
Using Proposition 5.2  in \cite{nv2}  we may assume  without loss of generality that $\epsilon=1.$  Observe  that the
$(2n-1)$-dimensional Lebesgue  measure  of
\begin{equation*}
 \big ( (A\cap\Delta_{a_0}(\rho))\times  \Delta_0^{n-1}(r')\big)\setminus C
\end{equation*}
is  zero  and that   the set  $ (A\cap\Delta_{a_0}(\rho))\times \Delta_0^{n-1}(r')$ is  living  on the boundary  of
the  smooth hypersurface  $\partial E\times  \Delta_0^{n-1}(r')$ in $\C^n.$ Consequently,  in the  desired   equality
we  may  suppose that $C=(A\cap\Delta_{a_0}(\rho))\times  \Delta_0^{n-1}(r').$ Then the
equality  follows easily from the  product  property for the extremal  function.
 \end{proof}

We come back to the proof of the proposition. Using the above  lemma,  we get
\begin{align*}
\widehat{ Y_n}&=\{(z,w',w_n)\in\Delta_{a_0}(\rho,\epsilon)\times\Delta_0^{n-1}(r') \times\Delta_0(r):\\ &
\omega\big((z,w'), C,\Delta_{a_0}(\rho,\epsilon)\times\Delta_0^{n-1}(r')\big)+
\omega\big(w_n,\Delta_0(r'),\Delta_0(r)\big)<1\}\\
&=\{(z,w',w_n)\in\Delta_{a_0}(\rho)\times\Delta_0^{n-1}(r')\times\Delta_0(r):\\ &  \omega\big(  (z,w'),
(A\cap\Delta_{a_0}(\rho))\times\Delta_0^{n-1}(r'), \Delta_{a_0}(\rho,\epsilon)\times\Delta_0^{n-1}(r') \big)+
\omega(w_n,\Delta_0(r'),\Delta_0(r)\big)<1\}\\
&=\{(z,w',w_n)\in\Delta_{a_0}(\rho,\epsilon)\times\Delta_0^{n-1}(r')\times\Delta_0(r):\\ &\quad \max\{
\frac{1}{\epsilon}\cdot  \omega\big(z, A\cap\Delta_{a_0}(\rho),\Delta_{a_0}(\rho)\big),\; \omega(w',
\Delta_0^{n-1}(r'),\Delta_0^{n-1}(r'))\}
\\&+\omega(w_n,\Delta_0(r'),\Delta_0(r))<1\}\\
&=\{(z,w',w_n)\in\Delta_{a_0}(\rho)\times\Delta_0^{n-1}(r')\times\Delta_0(r):\\ &  \frac{\omega(z,
A\cap\Delta_{a_0}(\rho),\Delta_{a_0}(\rho)\cap E)}{\epsilon}+ \omega(w_n,\Delta_0(r'),\Delta_0(r))<1\}.
\end{align*}

Since $r''<r$, we may find an $\rho_n\in(0,\rho]$ such that every function $f$  as in the hypotheses of Theorem
\ref{new_Imomkulov_thm}   extends holomorphically to a function $\widehat{ f}_n$ defined on
$\Delta_{a_0}(\rho_n)\times\Delta_0^{n-1}(r')\times\Delta_0(r'')\setminus T_n$. We may assume that $T_n$ is singular
with respect to the family $\{\widehat{ f}_n:\  f\  \text{as in  the hypotheses of  Theorem
\ref{new_Imomkulov_thm}}\}$.

Repeating the above argument for the coordinates $w_\nu$, $\nu=1,\dots,n-1$, and gluing the obtained sets, we find an
$\rho_0\in(0,\rho],$ $\epsilon_0\in (0,\epsilon]$ and a relatively closed pluripolar set $T_0:=\bigcup_{j=1}^nT_j$
such that every function $f$ as in  the hypotheses of  Theorem  \ref{new_Imomkulov_thm}  extends holomorphically to a
function $\widehat{ f}_0:=\bigcup_{j=1}^n\widehat{ f}_j$ holomorphic in
$\Delta_{a_0}(r_0,\epsilon_0)\times\Omega\setminus T_0$, where
\begin{equation*}
\Omega:=\bigcup_{j=1}^n \Delta_0^{j-1}(r')\times\Delta_0(r'')\times\Delta_0^{n-j}(r').
\end{equation*}
Let $\widehat{\Omega}$ denote the envelope of holomorphy of $\Omega$. Applying Theorem \ref{Chirka_thm}, we find a
relatively closed pluripolar subset $T$ of $\Delta_{a_0}(\rho_0,\epsilon_0)\times\widehat{\Omega}$ such that every
function $f$ as in  the hypotheses of  Theorem  \ref{new_Imomkulov_thm}  extends to a function $\widehat{ f}$
holomorphic on $\big (\Delta_{a_0}(\rho_0,\epsilon_0)\times\widehat{\Omega}\big)\setminus T$. Let
$r''':=\root{n}\of{{r'}^{n-1}r''}$. Observe that $\Delta_0(r''')\subset\widehat{\Omega}$. Recall that $r'''>r'_0$. We
may assume that $T$ is singular with respect to the family $\{\widehat{f}: \ f\   \text{as in  the hypotheses of
Theorem  \ref{new_Imomkulov_thm} }   \}$. Hence,  the proof is  finished.
\end{proof}

Now  we are in the position  to show that   Theorem \ref{local_thm_1} for $n=1$  implies  Theorem
\ref{new_Imomkulov_thm}.

\smallskip

\noindent{\bf  Proof of Theorem    \ref{new_Imomkulov_thm}.}  Suppose  without loss of generality that  all points of
$A$ are density points of $A.$ Using  a classical exhaustion  argument it suffices   to prove the following

\noindent{\bf Assertion.} {\it For every compact  set $A_0\subset A$ and  every $r^{'}\in (1,r),$ there  exist
$0<\rho=\rho_{r^{'}}$  and   a   relatively closed  pluripolar set $T=T_{r^{'}}\subset \widehat{\X}\big(A_0,E^n;E,
\Delta_0^n(r^{'})\big)$   such that every function $f$ satisfying   the hypotheses of Theorem
 \ref{new_Imomkulov_thm} extends holomorphically to  $ \widehat{\X}\big(A_0,E^n;E,   \Delta_0^n(r^{'})\big)
\setminus T.$ }

Now  fix a compact  set $A_0\subset A$ and  an $r^{'}\in (1,r).$ Applying  Proposition \ref{prop_local_extension} to
all points  of $A_0$ and  using the compactness of $A_0,$ we may find    $k$ points $a_1,\ldots, a_k\subset A$  and
$2k$  numbers $0<\rho_1, \epsilon_1,\ldots,\rho_k,\epsilon_k<1$ and  a   relatively closed  pluripolar set
$T^{'}\subset\Omega \times  \Delta_0^n(r^{'})$  with   $\Omega:= \bigcup^{k}_{j=1}\Delta_{a_j}(\rho_j,\epsilon_j)$
such that
\begin{itemize}
\item[$\bullet$] $A_0\subset \bigcup^{k}_{j=1}\Delta_{a_j}(\rho_j)$ \item[$\bullet$]
 every function $f$ satisfying   the hypotheses of Theorem
 \ref{new_Imomkulov_thm} extends holomorphically to  $\big (\Omega\times  \Delta_0^n(r^{'})\big)
\setminus T^{'}.$
\end{itemize}
We are able to define a new function $\tilde{f}$ on $\X\left(\Omega,E^n;E,\Delta_0^n(r^{'})
 \right)\setminus  T^{'}$ as follows
\begin{equation}\label{eq_Step1_formula_tildefdelta}
 \tilde{f}_{\delta}(z,w):=
\begin{cases}
 \hat{f}(z,w)
  & \qquad (z,w)\in  (\Omega \times \Delta_0^n(r^{'}))\setminus T^{'}, \\
  f(z,w) &   \qquad (z,w)\in (E\times E^n)        .
\end{cases}
\end{equation}
  Using the hypotheses on $f$ and the previous  argument, we see  that
  $$\tilde{f}_{\delta}\in \Oc_s\Big(
\X\left(\Omega, E^n;E, \Delta_0^n(r^{'})
 \right)\setminus T^{'},\C\Big).$$
   Consequently,
 $\tilde{f}$ satisfies the hypotheses of Theorem
 \ref{thm_Jarnicki_Pflug_new_version}.
 Applying  this theorem  yields a  relatively closed  pluripolar   subset $T$  of
$\widehat{\X}\left(\Omega, E^n; E,  \Delta_0^n(r^{'})      \right)$ with   $T\cap (E\times E^n)=\varnothing $ and a
function  $\hat{f}\in
 \Oc\Big(\widehat{\X}\left(   \Omega, E^n;E, \Delta_0^n(r^{'})
 \right)\setminus  T ,\C \Big)$
 such that  $\hat{f}=f$ on $E\times E^n .$
Using the  above-listed properties of $a_1,\ldots,a_k,$ we see that
 \begin{equation*}
\widehat{\X}\big(A_0,E^n;E,   \Delta_0^n(r^{'})\big)\subset \widehat{\X}\left(\Omega, E^n;E,  \Delta_0^n(r^{'})
\right).
\end{equation*}
This  proves the above  assertion, and thereby completes  the theorem.
 \hfill
$\square$

\smallskip

\section{Using holomorphic discs}
\label{Section_holomorphic_discs}

In this  section we combine  Poletsky's theory of discs \cite{po1,po2}, Rosay's Theorem on holomorphic discs \cite{ro}
and Theorem \ref{thm_Jarnicki_Pflug}.

Let us  recall some facts from Poletsky's theory of discs. For a complex manifold $\mathcal{M},$ let
$\mathcal{O}(\overline{E},\mathcal{M})$ denote the set of all holomorphic mappings $\phi:\ E\longrightarrow
\mathcal{M}$ which extend holomorphically  to   a neighborhood of  $\overline{E}.$ Such a mapping $\phi$ is called a
{\it holomorphic disc} on $\mathcal{M}.$ Moreover, for a subset $A$ of $\mathcal{M},$ let
\begin{equation*}
 1_{  A}(z):=
\begin{cases}
1,
  &z\in   A,\\
 0, & z\in \mathcal{M}\setminus A.
\end{cases}
\end{equation*}

In the work \cite{ro}  Rosay proved the following  result.
\begin{theorem}\label{Rosaythm}
Let $u$ be an upper semicontinuous function on a complex manifold $\mathcal{M}.$ Then the Poisson functional of $u$
defined by
\begin{equation*}
\mathcal{P}[u](z):=\inf\left\lbrace\frac{1}{2\pi}\int\limits_{0}^{2\pi} u(\phi(e^{i\theta}))d\theta:  \ \phi\in
\mathcal{O}(\overline{E},\mathcal{M}), \ \phi(0)=z \right\rbrace,
\end{equation*}
is plurisubharmonic on $\mathcal{M}.$
\end{theorem}
This  implies the following important consequence (see, for example, Proposition 3.4 in \cite{nv1}).
\begin{corollary}\label{Rosaycor}
Let $\mathcal{M}$ be a complex manifold  equipped  with the canonical  system of approach regions and   $A$   a
nonempty open subset of $\mathcal{M}.$    Then
   $ \omega(z,A,\mathcal{M}) = \mathcal{P}[1_{\mathcal{M}\setminus A}](z),$ $z\in\mathcal{M}.$
\end{corollary}

The main result of this  section is

\begin{theorem}\label{thm_discs}
  Let $X,\ Y$  be two complex manifolds,
  let $D\subset X,$ $ G\subset Y$ be two connected  open sets, let
  $A\subset D,$  $B\subset G$ be two non-empty open subsets.
  Let $M$ be  a  relatively closed subset of  $W:=\X(A,B;D,G)$  such that  $M$ is  locally
pluripolar in fibers over  $A$ and  over $B.$
  Then  there exists  a relatively closed locally pluripolar subset
 $\widehat{M}$ of $\widehat{W}$ such that $\widehat{M}\cap W\subset M$ and that
for  every mapping    $f\in \Oc_s(W\setminus M,Z),$
     there exists  a unique  mapping
$\hat{f}\in\Oc(\widehat{W}\setminus \widehat{M} ,Z)$ such that $\hat{f}=f$ on
  $ W\setminus M  .$
\end{theorem}
\begin{proof}
First shall prove the following  weaker version of  Theorem \ref{thm_discs}:

\noindent{\bf Assertion.} {\it For every $(z_0,w_0)\in\widehat{W},$ there  are a connected  open neighborhood $U\times V$ of
$(z_0,w_0)$ in $\widehat{W}$ and  a  relatively closed locally pluripolar  subset $S$ of $U\times V$  and a mapping
$\hat{f}\in\mathcal{O}\big((U\times V)\setminus S,Z\big)$  with $U\cap A\neq\varnothing\neq V\cap B$
such that $\hat{f}=f$ on $W\setminus (M\cup S).$}

Taking  for granted  this assertion, the theorem  follows immediately  from  a routine gluing process. Now  we  shall
present the proof of the assertion. Applying Theorem \ref{Rosaythm}  and Corollary \ref{Rosaycor}, we may find
$\phi\in \Oc(\overline{E},D)$ and $\psi\in \Oc(\overline{E},G)$
 such that
 \begin{equation}\label{eq_thm_discs_z0w0}
 \phi(0)=z_0,\quad \psi(0)=w_0,\qquad \frac{1}{2\pi}\Big(  \int\limits_{0}^{2\pi} 1_{D\setminus A}
 (\phi(e^{i\theta}))d\theta +  \int\limits_{0}^{2\pi} 1_{G\setminus B}
 (\psi(e^{i\theta}))d\theta \Big)<1.
 \end{equation}

Recall some ideas  in the   work of Rosay \cite[p. 166]{ro}.
 First we construct some kind of Hartogs figure $H$ with gaps which contains the image by a holomorphic embedding of
  a neighborhood $\{\phi(t):\ t\in\overline{E}\}$ in $D$   such that  $H$ is contained in a complex  manifold $D^{'}$
  spreading  over $\C^2\times D.$
 Composing with the projection $\Pi:\  \C^2\times D\rightarrow D,$ we get a natural map $\widetilde{\Pi}:\ D^{'}
 \rightarrow D.$ Let $\mathcal{U}^{'}$  be a Stein neighborhood of $H$ in $D^{'}.$ There exists  a proper
 holomorphic  embedding $\tau$ of $\mathcal{U}^{'}$ into $\C^{\mu}$ (for some $\mu$), and $\tau(\mathcal{U}^{'})$
 has  a neighborhood $\widetilde{\mathcal{U}}$ in $\C^{\mu}$ with a holomorphic  retract $r$ from
 $\widetilde{\mathcal{U}}$ onto
$\tau(\mathcal{U}^{'}).$ Consequently, setting $\Phi:=\widetilde{\Pi}\circ \tau^{-1}\circ r$ we obtain a surjective
mapping $\Phi\in\Oc(\widetilde{\mathcal{U}},\mathcal{U}),$ where $\mathcal{U}$ is a connected open neighborhood of
 $\{\phi(t):\ t\in\overline{E}\}$ in
$D.$  Analogously,
 we may find  a connected  open
 neighborhood $\mathcal{V}$ of $\{\psi(t):\
t\in\overline{E}\}$ in $G,$   an open subset $\widetilde{\mathcal{V}}$ in $\C^{\nu}$ (for some $\nu$)
 and  a surjective mapping
$\Psi\in\Oc(\widetilde{\mathcal{V}},\mathcal{V}).$

Next, consider the cross
\begin{equation*}
\mathcal{W}:=\X\big(\Phi^{-1}(A),\Psi^{-1}(B);\widetilde{\mathcal{U}},\widetilde{\mathcal{V}}\big)
\end{equation*}
and the set
\begin{equation*}
\mathcal{M}:=\left\lbrace (x,y)\in \widetilde{\mathcal{U}} \times \widetilde{\mathcal{V}}:\ (\Phi(x),\Psi(y))\in M
\right\rbrace.
\end{equation*}
Since  $\Phi$ and $\Psi$ are surjective, it is  clear that $\mathcal{M}$ is a relatively closed subset
of $\mathcal{W} $ and that $\mathcal{M}$ is locally pluripolar in
fibers over $\Phi^{-1}(A)$ and $\Psi^{-1}(B).$ Now  consider the mapping $F:\
\mathcal{W}\setminus\mathcal{M}\rightarrow Z$ defined by
\begin{equation*}
F(x,y):=f(\Phi(x),\Psi(y)),\qquad (x,y) \in  \mathcal{W}\setminus\mathcal{M}.
\end{equation*}
Using  the hypotheses of the theorem, we are able to apply Theorem \ref{thm_Jarnicki_Pflug} to $F.$ Consequently, we
obtain a relatively closed locally pluripolar subset $\widehat{\mathcal{M}}$ of  $\widehat{   \mathcal{W}}$ and a mapping
$\widehat{F}\in  \Oc(\widehat{   \mathcal{W}}\setminus\widehat{\mathcal{M}},Z)$  such that
$\mathcal{W}\cap \widehat{\mathcal{M}}\subset \mathcal{M} $ and $ \widehat{F}=F$ on $
\mathcal{W}\setminus \widehat{\mathcal{M}}.$

 Let $\widetilde{\mathcal{C}}$ be the set of critical points of $(x,y)\mapsto (\Phi(x),\Psi(y)).$ This is   a proper
analytic subset of $ \widetilde{\mathcal{U}} \times\widetilde{\mathcal{V}}$ since  $\Phi$ and $\Psi$ are surjective
(using Sard Theorem).
Now  define  the set
\begin{equation*}
\mathcal{C}:=\left\lbrace (\Phi(x),\Psi(y)):\ (x,y)\in  \widetilde{\mathcal{C}}   \right\rbrace.
\end{equation*}
It is not difficult to show that $\mathcal{C}$ is relatively closed   and is  contained in  a proper analytic subset of $\mathcal{U}\times\mathcal{V}.$

 Using the above formula for $F$ we  see that
 \begin{equation*}
 F(x,y)=F(x^{'},y^{'}),\qquad  \forall (x,y),(x^{'},y^{'})\in  \mathcal{W}\setminus\mathcal{M}:\
 \Phi(x)=\Phi(x^{'}),\ \Psi(y)=\Psi(y^{'}).
 \end{equation*}
 Consequently, using the Uniqueness Principle we can show that
  \begin{equation*}
 \widehat{F}(x,y)=\widehat{F}(x^{'},y^{'}),\qquad  \forall (x,y),(x^{'},y^{'})\in  \widehat{\mathcal{W}}\setminus\widehat{M}:\
 \Phi(x)=\Phi(x^{'}),\ \Psi(y)=\Psi(y^{'}).
 \end{equation*}
 Therefore, we can define the mapping
 \begin{equation*}
 \hat{f}(z,w):=\widehat{F}(\Phi^{-1}(z),\Psi^{-1}(w)),\qquad  (z,w)\in (\Phi,\Psi)(\widehat{\mathcal{W}}\setminus \widehat{\mathcal{M}}) \setminus S.
 \end{equation*}
 Since    $\Phi$ and $\Psi$  look like fibrations outside  $\mathcal{C},$  $\hat{f}$  is  holomorphic.
 Now  letting
 \begin{equation*}
 S:= \mathcal{C}\cup (\Phi,\Psi)(\widehat{\mathcal{M}}),
 \end{equation*}
 we see that  $S$ is  locally pluripolar in $\mathcal{U}\times\mathcal{V}.$
 Using (\ref{eq_thm_discs_z0w0})
we may choose  a  connected open neighborhood $U\times V$ of $(z_0,w_0)$ in
 $(\mathcal{U}\times\mathcal{V})\cap \widehat{W}$   with the required properties of the assertion.
  This completes the proof.
\end{proof}

The  following   result is  an immediate  consequence  of the above theorem.
\begin{corollary}\label{cor_local_thm_1}
Theorem \ref{local_thm_1}  still holds  in the following  general settings:
    $ Y$ is an arbitrary complex  manifold
and $G\subset Y$  is  an arbitrary domain
 and $B\subset G$ is an arbitrary  open subset.
\end{corollary}
\begin{proof}
Observe  that  the proof of Theorem \ref{local_thm_1} still works  if in the original hypothesis of the latter
theorem
 we only change
the following:   $B$  is  a polydisc, that is, $B$ is  not necessarily centered  at the center of the polydisc $G.$
The second step  will be   to require  that $B$ is   (only!) an open subset of the polydisc $G.$ For this  step we
should apply Theorem   \ref{thm_discs} in order to obtain  local  extensions. Then by a routine patching process, we
may  obtain  the global extension  from the local ones. The last step  will be  to require simply that  $B$ is an
open set  of an arbitrary  domain  $G.$ In fact, this  step is reduced  to the  second one by using  parameterized
families of holomorphic  discs  (see Lemma 3.2 in \cite{nv1}).
\end{proof}

\begin{corollary}\label{cor_local_thm_2}
Theorem \ref{local_thm_1}  still holds  in the following  general settings:
    $ Y$ is an arbitrary complex  manifold,
 $G\subset Y$  is  an arbitrary domain,
  $B\subset G$ is an arbitrary  open subset, $M$ is  not necessarily relatively closed in $W$  but
we require  instead  that  $A\subset D.$
\end{corollary}
\begin{proof}
Arguing  as in the proof of Corollary \ref{cor_local_thm_1}, we only need to  prove
 Theorem \ref{local_thm_1} when $A\subset D$ and $M$ is   not necessarily relatively closed in $W.$
But this follows readily from Theorem  \ref{thm_Jarnicki_Pflug}.
\end{proof}

\section{Proof of the Main Theorem}

First  we  will prove the following local version of the  Main Theorem.

\begin{theorem}\label{local_thm}
We keep the hypotheses  and  notation of the Main Theorem.  Let  $A_0$ (resp. $B_0$) be  a subset of
 $\overline{ D}$ (resp. $\overline{ G}$) such that $A_0$ and  $B_0$ are locally pluriregular
    and that
$\overline{A}_0\subset A$ and $\overline{B}_0\subset B.$
 Then for every $(a,b)\in \overline{A_0}\times \overline{B_0},$ there exist an
open neighborhood $U$ of $a$ in $X,$  an open neighborhood $V$ of $b$ in $Y$,  and a relatively closed locally
pluripolar subset
 $\widehat{M}=\widehat{M}_{(a,b)}$ of $\widehat{\X}\big (A_0\cap U, B_0\cap V; D, G\cap V     \big)$ such that:
 \begin{itemize}
 \item[$\bullet$]
 $ \X\big (A_0\cap U, B_0\cap V; D, G \cap V   \big)\setminus M
 \subset \End\big(\widehat{\X}\big (A_0\cap U, B_0\cap V; D, G\cap V     \big)\setminus \widehat{M}\big);
 $
 \item[$\bullet$]
 $ \X\big (A_0\cap U, B_0\cap V; D, G \cap V   \big)\cap\widehat M\subset M;$
\item[$\bullet$]
 for  every mapping    $f:\ W\setminus M\longrightarrow Z$
   satisfying conditions (i)--(iii) of the Main Theorem,
     there exists  a unique  mapping
$\hat{f}\in\Oc\big( \widehat{\X}\big (A_0\cap U, B_0\cap V; D, G\cap V    \big)        \setminus \widehat{M}
,Z\big)$  such that $\widehat f$ admits the $\mathcal{A}$-limit
$f(\zeta,\eta)$ at every point
  $$(\zeta,\eta)\in   \X\big (A_0\cap U, B_0\cap V; D, G \cap V   \big)\setminus M  .$$
\end{itemize}
\end{theorem}
\begin{proof}
There  are two cases  to consider.

\smallskip

\noindent{\bf Case}  $(a,b)\not\in D\times G.$

      Invoking the hypothesis  on $M$
we see that  there exist an open neighborhood $U$ of $a$ in $X$ and  an open neighborhood $V$ of $b$ in $Y$  such
that
\begin{equation*}
\X\big(A\cap U,B\cap V;D\cap U,G\cap V\big)\cap M=\varnothing.
\end{equation*}
Moreover, we may assume without loss  of generality that  $U$ and $V$ are biholomorphic to some bounded Euclidean
domains. Using this we are able to apply Theorem \ref{Nguyen_thm} to $f$ restricted to $\X\big(A\cap U,B\cap V;D\cap
U,G\cap V\big).$ Consequently,
 we obtain  a mapping  $\hat{f}_0\in \Oc\big(
\widehat{\widetilde\X}(A\cap U,B\cap V;D\cap U,G\cap V),Z\big )$ which extends  $f.$  For $0<\delta<1$  let
\begin{eqnarray*}
A_{\delta}&:=&\left\lbrace z\in U\cap D:\  \omega(z,A_0\cap U,D\cap U)<\delta    \right\rbrace,\\ V_{\delta}&:=&
\left\lbrace w\in V\cap G:\  \omega(w,B_0\cap V,G\cap V)<1-\delta    \right\rbrace,\\ B_{\delta}&:=& V_{1-\delta}.
\end{eqnarray*}

Now let $0<\delta <\frac{1}{2}.$ Then  by the above discussion
  $\hat{f}_0$ is  holomorphic on   $A_{\delta}\times V_{\delta}.$
On the other hand, by the hypotheses $f(\cdot,w)$ is  holomorphic on $D\setminus M^w$ for all $w\in B\cap V.$
Therefore, we   are in the position to apply Corollary \ref{cor_local_thm_1} to the following mapping
 $f_{\delta}:\ \X\left(A_{\delta}, B\cap V;D,  V_{\delta})
 \right)\setminus M  ) \rightarrow  Z$  given by
\begin{equation*}
 f_{\delta}(z,w):=
\begin{cases}
 \hat{f}_0(z,w),
  & \qquad (z,w)\in  A_{\delta}\times V_{\delta}  , \\
  f(z,w), &   \qquad (z,w)\in \big(D\times (B\cap V)\big)\times \setminus M      .
\end{cases}
\end{equation*}
Consequently, we obtain a relatively
 closed locally pluripolar subset $\widehat{M}_{\delta}$ of
$$ \widehat{\X}\left(A_{\delta}, B_0\cap V;D, V_{\delta}
 \right)$$ and a mapping
 $$\hat{f}_{\delta}\in \Oc\Big(  \widehat{\X}\left(   A_{\delta}, B_0\cap V;D,V_{\delta}         \right) \setminus
 \widehat{M}_{\delta}  ,Z\Big)
 $$ which extends $f.$ Since  $\omega(\cdot,A_{\delta},D)\leq  \omega(\cdot,A_0\cap U,D)$ on $D,$
 it follows that
 $$  \widehat{\X}\left(A_0\cap U, B_0\cap V;D,V_{\delta}
 \right)\subset \widehat{\X}\left(A_{\delta}, B_0\cap V;D, V_{\delta}
 \right).
 $$
On the other hand,  by  Proposition 3.5 in \cite{nv2},
 \begin{equation*}
  \widehat{\X}\left(A_0\cap U, B_0\cap V;D,V_{\delta}
 \right)  \nearrow \widehat{\X}\left(A_0\cap U, B_0\cap V;D, G\cap V
 \right)\  \text{as}\  \delta\searrow 0.
  \end{equation*}
  Now  fix a sequence  $(\delta_k)_{k=1}^{\infty}$ such that  $0<\delta_{k}<1$ and $\delta_k\searrow 0^{+}.$
  Therefore, using the last  equality, we may glue  $(\widehat{M}_{\delta_k})_{k=1}^{\infty}$  (take again the
  smallest singular sets)  together in order to
  obtain  a  relatively closed  pluripolar   subset $\widehat{M}$  of  $  \widehat{\X}\left(A_0\cap U, B_0\cap V;D,
  G\cap V
 \right)               $
and an extension mapping $\hat{f}$ holomorphic on $  \widehat{\X}\left(A_0\cap U, B_0\cap V;D, G\cap V
 \right)\setminus \widehat{M}$ with the desired  properties of the theorem.

\noindent{\bf Case} $(a,b)\in D\times G.$

Choose  an open neighborhood $U\subset D$ of $a$ (resp. $V\subset G$ of $b$) which is  biholomorphic to   a bounded
Euclidean domain. Using the hypotheses , we are able to apply
 Theorem \ref{thm_Jarnicki_Pflug} to $f|_{\X(A\cap U,B\cap V;U,V)}.$
Consequently, we obtain a relatively  closed locally pluripolar  subset $\widehat{M}_0$ of
 $\widehat{\X}(A_0\cap U,B_0\cap V;U,V)$  and  a mapping  $\hat{f}_0\in \Oc\big(
\widehat{\X}(A_0\cap U,B_0\cap V;U,V)\setminus  \widehat{M}_0,Z\big )$ which extends  $f.$ The remaining part of the proof follows  along the same lines as those given in the previous
case.  The only difference  is that  we  will  apply Corollary \ref{cor_local_thm_2} instead of Corollary
\ref{cor_local_thm_1}.
\end{proof}

Finally, we arrive  at the

\smallskip

\noindent {\bf Proof of the Main Theorem.} By  Proposition \ref{Nguyen_prop}, we only need to check  the condition
stated  in that  proposition. In the sequel we  are under  the hypotheses and notation  introduced  in that condition.
The proof will be divided into two steps.

 \noindent{\bf Step 1.}
{\it  Under the hypothesis  and notation of Part 1) of Proposition 3.11, let $b_0\in \overline{B}_0.$   Then there
exists an open set neighborhood $V$ of $b_0$ in $Y$  and  a relatively closed locally  pluripolar subset $\widehat{M}$
of $\widehat{\X}(A_0,B_0\cap V;D,V\cap G)$
such that:
 \begin{itemize}
 \item[$\bullet$]
 $ \X\big (A_0, B_0\cap V; D, G \cap V   \big)\setminus M
 \subset \End\big(\widehat{\X}\big (A_0, B_0\cap V; D, G\cap V     \big)\setminus \widehat{M}\big);
 $
 \item[$\bullet$]
 $ \X\big (A_0, B_0\cap V; D, G \cap V   \big)\cap\widehat M\subset M;$
\item[$\bullet$]
 for  every mapping    $f:\ W\setminus M\longrightarrow Z$
   satisfying conditions (i)--(iii) of the Main Theorem,
     there exists  a unique  mapping
$\hat{f}\in\Oc\big( \widehat{\X}\big (A_0, B_0\cap V; D, G\cap V    \big)        \setminus \widehat{M}
,Z\big)$  such that $\widehat f$ admits the $\mathcal{A}$-limit
$f(\zeta,\eta)$ at every point
  $$(\zeta,\eta)\in   \X\big (A_0, B_0\cap V; D, G \cap V   \big)\setminus M  .$$
\end{itemize}
}

 Since
$\overline{A}_0$ is compact, we may apply Theorem \ref{local_thm} to all pairs $(a,b_0),$  $a\in\overline{A}_0.$
Consequently, we may find a finite number of points $a_1,\ldots,a_M,$  their respective open neighborhoods
$U_{1},\ldots,U_{M}$ in $X$ and an open neighborhood $V$ of $b_0$ in $Y$ with the following properties:
 \begin{itemize}
\item[$\bullet$] $\overline{A}_0\subset\bigcup\limits_{j=1}^MU_{j};$
 \item[$\bullet$] there exist  a
relatively  closed locally  pluripolar subset
 $\widehat{M}_{j}$ of $ \widehat{\X}\big(A_0\cap U_{j}, B_0\cap V; D,G\cap V\big)         $ and a   mapping
$\hat{f}_j\in\Oc\big( \widehat{\X}\big(A_0\cap U_{j}, B_0\cap V; D,G\cap V\big)              \setminus
\widehat{M}_{j} ,Z\big)$ which admits the $\mathcal{A}$-limit $f(\zeta,\eta)$ at every point
  $(\zeta,\eta)\in  \X\big(A_0\cap U_{j}, B_0\cap V; D,G\cap V\big)\setminus M  .$
\end{itemize}

For  any  $0<\delta<\frac{1}{2}$ put
\begin{eqnarray*}
B_{\delta} &:=& \{w\in G\cap V:\  \omega(w,B_0\cap V, G\cap V)<\delta\},\qquad V_{\delta}:=B_{1-\delta},\\
A_{j,\delta}&:=&\{z\in D\cap U_{j}:\  \omega(z,A_0\cap U_{j}, D)<\delta\},\qquad j=1,\ldots,M;
\\
D_{j,\delta}&:=&A_{j,1-\delta},\qquad A_{\delta}:= \bigcup\limits_{j=1}^M A_{j,\delta},\qquad D_{\delta}:=
\bigcup\limits_{j=1}^MD_{j,\delta}.
\end{eqnarray*}
Observe that $A_{\delta}$ and $D_{\delta}$  are open subsets of $D,$ and $B_{\delta}$ and $V_{\delta}$  are open
subsets of $V.$ Moreover,  $D_{\delta}\nearrow D,$    $V_{\delta}\nearrow V$ and
\begin{equation}\label{eq_limit_prop_semi_local}
\omega(\cdot,A_0, D_{\delta})\searrow \omega(\cdot,A_0, D)\quad\text{and}\quad\omega(\cdot,B_0\cap V,
V_{\delta})\searrow \omega(\cdot,B_0, V)
\end{equation}
 as $\delta\searrow 0.$ Using the above  constructions, we may glue the
 sets $\widehat{M}_{j}$ (resp. the  mappings $\big( \hat{f}_{j}\big)_{j=1}^M$) together  in order to obtain   a
relatively  closed locally pluripolar subset
 $M_{\delta}$ of $\X(A_{\delta},B_{\delta};D_{\delta},V_{\delta})$ and a   mapping
$\widetilde{f}_{\delta}\in\Oc_s\big( \X(A_{\delta},B_{\delta};D_{\delta},V_{\delta})           \setminus M_{\delta}
,Z\big).$
 Applying Theorem \ref{thm_discs} to $\widetilde{f}_{\delta}, $  we  get   a relatively  closed locally
pluripolar subset
 $\widehat{M}_{\delta}$ of $\widehat{\X}(A_{\delta},B_{\delta};D_{\delta},V_{\delta})$ and a   mapping
$\widehat{f}_{\delta}\in\Oc\big( \widehat{\X}(A_{\delta},B_{\delta};D_{\delta},V_{\delta}) \setminus
\widehat{M}_{\delta} ,Z\big).$ Now  using (\ref{eq_limit_prop_semi_local}), we may glue the mappings $\big(
\hat{f}_{\delta}\big)_{0<\delta<1}$ together in order to obtain a relatively closed locally pluripolar subset
 $\widehat{M}$ of $\X(A_0,B_0\cap V ;D,V)$ and a   mapping
$\widehat{f}\in\Oc\big( \widehat{\X}(A_0,B_0\cap V ;D,V) \setminus \widehat{M} ,Z\big).$ Hence, Step 1 is  complete.

\smallskip

\noindent{\bf Step 2.  } {\it End  of the proof.}

Note that $\overline B_0$ is compact. Applying  step 1 we may now proceed  in exactly the same way as we did in step 1
starting from Theorem 5.1. Consequently,  Step 2 follows. Details are left to the interested reader.

Hence the condition in Proposition \ref{Nguyen_prop} are verified and, therefore, the proof of the Main Theorem is
complete.

\section{Applications}  \label{section_applications}

In \cite{nv2} the first author gives various  applications of the Main Theorem  for the case where $M=\varnothing$  using three different systems of approach regions. These  are the canonical one,  the system of angular approach regions and the system of conical approach regions. We only give here  some applications of  the system of conical approach regions.   The reader may try to  treat the  first two cases, that is, to translate Theorem 10.2 and 10.3 of \cite{nv2} into the context of the Main Theorem.

Let $D\subset\C^n$ be a  domain  and $A\subset\partial D.$ We suppose in addition that
 $D$ is  {\it locally $\mathcal{C}^2$ smooth} on $A$ (i.e.
for any $\zeta\in A,$ there exist an open neighborhood $U=U_{\zeta}$ of
 $\zeta$ in $\C^n$ and a   real function
 $\rho=\rho_{\zeta}\in \mathcal{C}^2(U)$ such that $D\cap U=\lbrace z\in U:\
 \rho(z)<0\rbrace$ and $d\rho(\zeta)\not=0$).
 We  define {\it the system   of   conical approach regions supported on $A$:}
   $\mathcal{A}=\big(\mathcal{A}_{\alpha}(\zeta)\big)_{\zeta\in\overline{ D},\  \alpha\in I_{\zeta}}$
as  follows:
  \begin{itemize}
  \item[$\bullet$] If $\zeta\in  \overline{ D}\setminus A,$  then $\big(\mathcal{A}_{\alpha}(\zeta)\big)_{ \alpha\in I_{\zeta}}$
coincide with the  canonical    approach regions.
 \item[$\bullet$] If $\zeta\in  A,$ then
 \begin{equation*}
 \mathcal{A}_{\alpha}(\zeta):=\left\lbrace  z\in D:\  \vert  z-\zeta\vert <\alpha\cdot \dist(z,T_{\zeta})  \right\rbrace,
 \end{equation*}
where $I_{\zeta}:=(1,\infty)$  and $ \dist(z,T_{\zeta})$ denotes the Euclidean    distance  from the  point $z$ to the       to the tangent hyperplane $T_{\zeta}$  of $\partial D$ at $\zeta.$
\end{itemize}

We can also   generalize the previous construction to  a  global situation:

{\it  $X$ is an  arbitrary complex manifold, $D\subset  X$ is an open set and  $A\subset \partial D$ is  a  subset
with  the property that $D$ is   locally $\mathcal{C}^2$ smooth on $A.$
}

Let  $X$  be  an arbitrary  complex manifold  and $D\subset  X$  an open subset.
 We say  that  a set   $A\subset \partial D$ is {\it locally contained in a  generating manifold} if  there  exist  an
 (at most countable)
 index set $J\not=\varnothing,$ a  family  of open subsets
 $(U_j)_{j\in J}$  of $X$ and  a  family  of
   {\it generating manifolds} \footnote{ A differentiable  submanifold  $\mathcal{M}$ of a complex manifold
  $X$ is  said to be a {\it generating
  manifold}
  if  for all  $\zeta\in\mathcal{M},$   every  complex vector subspace of  $T_{\zeta}X$ containing
  $T_{\zeta}\mathcal{M}$
coincides  with
    $T_{\zeta}X.$}   $(\mathcal{M}_j)_{j\in J}$ such that
    $A\cap U_j\subset \mathcal{M}_j,$  $j\in J,$ and that $A\subset  \bigcup_{j\in J}  U_j.$
The dimensions of $\mathcal{M}_j$ may  vary according  to $j\in J.$

Suppose that  $A\subset \partial D$ is  locally contained in a generating manifold. Then we say that $A$ is {\it of
positive size} if under the above notation  $\sum_{j\in J}\mes_{ \mathcal{M}_j}(A\cap U_j)>0,$ where
$\mes_{\mathcal{M}_j}$ denotes the    Lebesgue measure on $\mathcal{M}_j.$
 A point $a\in A$ is  said  to be  a {\it density point} of $A$  if it is a density point of $A\cap U_j$ on
 $\mathcal{M}_j$ for some $j\in J.$
 Denote  by  $A^{'}$ the set of density points   of $A.$

 Suppose now that, in addition,  $A\subset \partial D$ is  of positive size.
 We  equip $D$  with the  system of conical approach regions   supported on   $A.$
  Using the   work of  B. Coupet (see  Th\'eor\`eme 2 in \cite{co}),
 one can show that\footnote{  A complete proof will be  available in \cite{nv3}.}   $A$  is  locally  pluriregular  at
 all density points
  of $A$ and $A^{'}\subset \widetilde{A}.$  Consequently, it  follows from  Definition
\ref{defi_pluri_measure}  that
\begin{equation*}
\widetilde{\omega}(z,A,D)\leq  \omega(z,A^{'},D),\qquad  z\in D.
\end{equation*}
This estimate, combined with the Main Theorem, implies the  following result.

  \begin{theorem} \label{application1_3}
  Let $X,\ Y$  be  two complex  manifolds,
  let $D\subset X,$ $ G\subset Y$ be  two connected open sets, and let
  $A$ (resp. $B$) be  a  subset of  $\partial D$ (resp.
  $\partial G$).
   $D$  (resp.  $G$) is  equipped with a
  system of conical approach  regions
  $\big(\mathcal{A}_{\alpha}(\zeta)\big)_{\zeta\in\overline{ D},\  \alpha\in I_{\zeta}}$
  (resp.  $\big(\mathcal{A}_{\beta}(\eta)\big)_{\eta\in\overline{ G},\  \beta\in I_{\eta}}$)  supported on  $A$
  (resp. on $B$).
   Suppose  in  addition that
    $A$  and  $B$  are  of positive size.
   Define
   \begin{eqnarray*}
   W^{'}  &:= &\X(A^{'},B^{'};D,G),\\
    \widehat{W^{'}}  &:= &\left\lbrace  (z,w)\in D\times G:\  \omega(z,A^{'},D)+\omega(w,B^{'},G)<1
    \right\rbrace,
   \end{eqnarray*}
     where $A^{'}$  (resp.  $B^{'}$) is  the set of  density  points   of $A$  (resp.  $B$).
Let $M$ be  a  relatively closed subset of  $W$  with the following properties:
\begin{itemize}
\item[$\bullet$] $M$ is  thin in  fibers  (resp.  locally   pluripolar in fibers) over  $A$ and  over $B;$
    \item[$\bullet$] $M\cap (A\times B)=\varnothing.$
\end{itemize}
Then  there exists a relatively closed  analytic  (resp. a  relatively  closed locally  pluripolar) subset
 $\widehat{M}$ of $\widehat{W^{'}}$ such that
for  every mapping    $f:\ W\setminus M\longrightarrow Z$
   satisfying the following  conditions:
   \begin{itemize}
   \item[(i)]    $f\in\Cc_s(W\setminus M,Z)\cap \Oc_s(W^{\text{o}}\setminus M,Z);$
    \item[(ii)] $f$ is locally bounded   along  $\X(A,B;D,G)
    \setminus M;$
    \item[(iii)]          $f|_{(A\times B)}$ is  continuous,
    \end{itemize}
     there exists  a unique  mapping
$\hat{f}\in\Oc(\widehat{W^{'}}\setminus \widehat{M} ,Z)$ which admits the $\mathcal{A}$-limit $f(\zeta,\eta)$ at
every point
  $(\zeta,\eta)\in   W^{'}\setminus M  .$
\end{theorem}

The second  application is  a very general mixed cross   theorem.

 \begin{theorem} \label{application2_3}
  Let $X,\ Y$  be  two complex  manifolds,
  let $D\subset X,$ $ G\subset Y$ be connected  open sets,  let
  $A$ be a subset of  $\partial D,$   and let  $B$  be  a subset of
  $ G$.
   $D$ is  equipped with the
  system of conical approach regions
  $\big(\mathcal{A}_{\alpha}(\zeta)\big)_{\zeta\in\overline{ D},\  \alpha\in I_{\zeta}}$  supported on  $A$
   and  $G$  is equipped with the canonical   system of
   approach  regions
  $\big(\mathcal{A}_{\beta}(\eta)\big)_{\eta\in\overline{ G},\  \beta\in I_{\eta}}$.
   Suppose in addition that
    $A\subset \partial D$  is  of positive size and that  $B=B^{\ast}\not=\varnothing.$
   Define
    \begin{eqnarray*}
   W^{'}  &:= &\X(A^{'},B;D,G),\\
    \widehat{W^{'}}  &:= &\left\lbrace  (z,w)\in D\times G:\  \omega(z,A^{'},D)+\omega(w,B,G)<1     \right\rbrace,
   \end{eqnarray*}
    where $A^{'}$  is the set of  density points  of $A.$
    Let $M$ be  a  relatively subset of  $W$  with the following properties:
\begin{itemize}
\item[$\bullet$] $M$ is  thin  in  fibers  (resp.  locally   pluripolar in fibers) over  $A$ and  over $B;$
 \item[$\bullet$] $M\cap (A\times B)=\varnothing.$
\end{itemize}
Then  there exists a relatively closed  analytic  (resp. a  relatively  closed locally  pluripolar) subset
 $\widehat{M}$ of $\widehat{W^{'}}$ such that $   W^{'}\setminus M\subset \End(\widehat{W^{'}}\setminus
 \widehat{M})$
 and that
for  every mapping    $f:\ W\setminus M\longrightarrow Z$
   satisfying the following  conditions:
   \begin{itemize}
   \item[(i)]    $f\in\Cc_s(W\setminus M,Z)\cap \Oc_s(W^{\text{o}}\setminus M,Z);$
    \item[(ii)] $f$ is locally bounded   along  $(A\times G)
    \setminus M,$
    \end{itemize}
     there exists  a unique  mapping
$\hat{f}\in\Oc(\widehat{W^{'}}\setminus \widehat{M} ,Z)$ which admits the $\mathcal{A}$-limit $f(\zeta,\eta)$ at
every point
  $(\zeta,\eta)\in    W^{'}\setminus M  .$
\end{theorem}

V.-A.  Nguy{\^e}n, School of Mathematics, Korea Institute  for Advanced Study, 207-43Cheongryangni-2dong,
Dongdaemun-gu, Seoul 130-722, Korea.

{\tt vietanh@kias.re.kr}

\

\noindent

Peter Pflug, Carl von  Ossietzky Universit\"{a}t Oldenburg, Institut f\"{u}r   Mathematik, Postfach 2503, D--26111,
 Oldenburg, Germany.

{\tt pflug@mathematik.uni-oldenburg.de}

\end{document}